\documentclass[12pt]{amsart}
\usepackage{color}
\usepackage{enumerate}
\usepackage{graphics}
\usepackage[colorlinks, urlcolor=blue]{hyperref}
\usepackage{mathrsfs}
\usepackage{amssymb}
\usepackage[all]{xy}
\usepackage{amsthm}

\newtheorem{theorem}{Theorem}[section]
\newtheorem*{theorem*}{Theorem}
\newtheorem*{problem*}{Problem}
\newtheorem{lemma}[theorem]{Lemma}
\newtheorem{proposition}[theorem]{Proposition}
\newtheorem{corollary}[theorem]{Corollary}
\newtheorem{definition}[theorem]{Definition}

\theoremstyle{remark}

\newtheorem{remark}[theorem]{Remark}

   \newcommand{\Rdb}{\mbox{$\mathbb{R}$}}

   \renewcommand{\H}{\mbox{${\mathcal H}$}}

   \newcommand{\N}{\mbox{${\mathbb N}$}}

   \def\R{{\mathbb{R}}}

   \renewcommand{\S}{\mbox{${\mathcal S}$}}
   \def\Sc{{\mathcal{S}}}

   \def\Im{{\mathrm{Im}\,}}
   
   \def\tr{{\mathrm{Tr}\,}}

\begin{document}

\title[Higher order $\Sc^2$-differentiability]
{Higher order $\Sc^2$-differentiability and application to Koplienko trace formula}

\author[C. Coine]{Cl\'ement Coine}
\email{clement.coine1@gmail.com $\,\&\,$ clement.coine@univ-fcomte.fr}
\author[C. Le Merdy]{Christian Le Merdy}
\email{clemerdy@univ-fcomte.fr}
\author[A. Skripka]{Anna Skripka}
\email{skripka@math.unm.edu}
\author[F. Sukochev]{Fedor Sukochev}
\email{f.sukochev@unsw.edu.au}

\address{C.C. : School of Mathematics and Statistics,
Central South University, Changsha 410085, People’s Republic of China
$\,\&\,$
Laboratoire de Math\'ematiques de Besan\c con, UMR 6623, CNRS, Universit\'e Bourgogne Franche-Comt\'e,
25030 Besan\c{c}on Cedex, France}
\address{C.L. : Laboratoire de Math\'ematiques de Besan\c con, UMR 6623, CNRS,
Universit\'e Bourgogne Franche-Comt\'e,
25030 Besan\c{c}on Cedex, France}
\address{A.S. : Department of Mathematics and Statistics, MSC01 1115,
University of New Mexico, Albuquerque, NM 87131, USA}
\address{F.S. : School of Mathematics and Statistics, University of NSW,
Kensington NSW 2052, Australia}



\date{\today}

\maketitle

\begin{abstract}
Let $A$ be a selfadjoint operator in a separable Hilbert space, $K$ a selfadjoint Hilbert-Schmidt operator, and $f\in C^n(\mathbb{R})$. We establish that $\varphi(t)=f(A+tK)-f(A)$ is $n$-times continuously differentiable on $\mathbb{R}$ in the Hilbert-Schmidt norm, provided either $A$ is bounded or the derivatives $f^{(i)}$, $i=1,\ldots,n$, are bounded. This substantially extends the results of \cite{ACDS} on higher order differentiability of $\varphi$ in the Hilbert-Schmidt norm for $f$ in a certain Wiener class. As an application of the second order $\Sc^2$-differentiability, we extend the Koplienko trace formula
from the Besov class $B_{\infty1}^2(\R)$ \cite{Peller2005} to functions $f$ for which the divided difference $f^{[2]}$ admits a certain Hilbert space factorization.
\end{abstract}

\section{Introduction}
Let $\mathcal{H}$ be a separable Hilbert space and let $\mathcal{S}^1(\mathcal{H})$, $\mathcal{S}^2(\mathcal{H})$, and $\mathcal{B}(\mathcal{H})$ denote respectively the trace class, the Hilbert-Schmidt class, and the space of bounded linear operators on $\mathcal{H}$. Let $A$ be a possibly unbounded selfadjoint operator densely defined in $\mathcal{H}$ and let $K$ be a selfadjoint operator in $\mathcal{S}^2(\mathcal{H})$. Let $f : \mathbb{R} \rightarrow \mathbb{C}$ be a Lipschitz function and let $\varphi$ denote the mapping given by
\begin{equation}\label{functionphi}
\varphi: \R\ni t\mapsto f(A+tK)-f(A).
\end{equation}
In this paper we address higher order continuous $\mathcal{S}^2(\mathcal{H})$-differentiability of $\varphi$ and its application in perturbation theory. The notion of differentiability we are dealing with is the classical definition of differentiability for Banach space valued functions defined on an open interval of $\mathbb{R}$.

Study of differentiability of operator functions was initiated in \cite{DK1956} and has been motivated by problems in perturbation theory. In \cite{DK1956}, existence of the $n$th order derivative $\varphi^{(n)}$ of the operator valued function given by \eqref{functionphi} was established in the operator norm for bounded $A$ and $f\in C^{2n}(\R)$.
Sharp sufficient conditions for existence of the first order derivative $\varphi'$ were established in \cite{Peller1990}
 and \cite{DeP-Suk,KPSS} in the operator norm and $\mathcal{S}^2(\mathcal{H})$-norm, respectively. In particular, the condition ``$f\in C^1(\R)$ and $f'$ is bounded'' is insufficient for existence of $\varphi'$ in the operator norm and sufficient for existence of $\varphi'$ in the $\mathcal{S}^2(\mathcal{H})$-norm. Existence of the $n$th order derivative $\varphi^{(n)}$  was established in the operator norm in \cite{Peller2006} for $f$ in the intersection of the Besov classes $B_{\infty1}^1(\R)\cap B_{\infty1}^n(\R)$ and in the symmetric operator ideal norm in \cite{ACDS} for $f$ in the Wiener class $W_{n+1}(\R)$. In \cite{ACDS,DK1956,DeP-Suk,KPSS,Peller1990,Peller2006}, the operator derivative $\varphi^{(n)}$ was represented as a certain multiple operator integral.

In our first main result (Theorem \ref{Formula}) we substantially enlarge the set of $f$ for which the higher order derivative $\varphi^{(n)}$ exists. We establish
that $\varphi$ is $n$-times continuously differentiable in $\mathcal{S}^2(\mathcal{H})$, provided $f\in C^n(\mathbb{R})$ and either $A$ is bounded or $f^{(i)}$, $i=1,\ldots,n$, are bounded. We derive the representation
\[\dfrac{1}{k!} \varphi^{(k)}(t) = \left[ \Gamma^{A+tK, A+tK, \ldots, A+tK}(f^{[k]}) \right] (K, \ldots, K),\quad k=1,\ldots,n,\] for derivatives $ \varphi^{(k)}$ via multiple operator integrals defined in \cite{CLS}, where $f^{[k]}$ is the $k$th order divided difference of $f$. The latter representation for $\varphi^{(k)}$ plays a key role in proving existence of $\varphi^{(k+1)}$ (induction process) and in the study of the operator Taylor remainder, which satisfies
\[f(A+K)-f(A)-\sum_{k=1}^{n-1}\dfrac{1}{k!}\varphi^{(k)}(0)=\left[\Gamma^{A+K,A,\ldots,A}(f^{[n]})\right](K,\ldots,K).\]
We note that such a formula already appeared either for special functions \cite[Theorem 5.1 $\&$ (5.1)]{AP} or in the finite
dimensional case \cite[Theorem 3.7]{PSST}.

One of the fundamental results in perturbation theory is the existence
of the Lifshits-Krein spectral shift function $\xi\in L^1(\R)$ uniquely
determined by self-adjoint operators $A$ and $K$ with $K\in \Sc^1(\H)$ and satisfying the trace formula
\begin{equation}\label{KreinsFormula}
\tr(f(A+K)-f(A))=\int_\mathbb R f'(t) \xi(t)\, dt
\end{equation}
for every $f\in C^1(\R)\cap L^1(\R)$ with $\widehat{f'}\in L^1(\mathbb R)$ \cite{K1,Lif}. We refer to \cite{Yafaev} for applications of the spectral shift function in mathematical physics and to \cite{ASsurvey} and references cited therein for generalizations of $\xi$. In \cite{Peller1990}, the trace formula \eqref{KreinsFormula} was extended to $f$ in the Besov class $B^1_{\infty 1}(\mathbb{R})$ and later in \cite{Peller2016}, to the class of differentiable functions $f$ with bounded $f'$ such that there exist a Hilbert space $H$ and bounded functions $\alpha,\beta:\R\to H$ (or a separable Hilbert space $H$ and bounded
weakly continuous functions $\alpha,\beta:\R\to H$) satisfying
\begin{equation}\label{HSF}
f^{[1]}(s,t)=\left<\alpha(s),\beta(t)\right>,\quad \forall (s,t) \in \mathbb{R}^2.
\end{equation}
According to \cite[Theorems 2.2.2, 2.2.3, 3.1.10, 3.3.6]{APeller}, the differentiability of $f$
and either of the above factorization properties
are equivalent to $f$ being $\S^1$-operator Lipschitz.

It was proved in \cite{Ko1} that there exists a function $\eta \in L^1(\mathbb{R})$ uniquely determined by $A$ and $K\in\Sc^2(\H)$, called the Koplienko spectral shift function associated to $(A,K)$, such that
\begin{equation}\label{Kop}
\tr\Big(f(A+K) - f(A) - \dfrac{d}{dt}f(A+tK)_{|t=0}\Big) = \int_{\mathbb{R}} f''(t) \eta(t) \, dt,
\end{equation}
for rational functions $f$ with nonpositive degree and poles off $\mathbb{R}$, where $\dfrac{d}{dt}f(A+tK)_{|t=0}$ is evaluated in the operator norm. In \cite{Peller2005}, the trace formula \eqref{Kop}
was extended to functions in the intersection of Besov classes $B^1_{\infty 1}(\mathbb{R})\cap B^2_{\infty 1}(\mathbb{R})$.
A modification of this formula also holds for $f\in B^2_{\infty 1}(\mathbb{R})$ (see Remark \ref{PellerMS} for the precise formulation).

In our second main result (Theorem \ref{ThmA}), we prove that
\begin{equation}
\label{generalTF}
\tr\left(\left[\Gamma^{A+K,A,A}(f^{[2]})\right](K,K)\right) = \int_{\mathbb{R}} f''(t) \eta(t) \, dt,
\end{equation}
for every $f\in C^2(\R)$ with bounded $f''$ and such that there exist a separable Hilbert space $H$ and two bounded Borel functions $a, b : \mathbb{R}^2 \rightarrow H$ satisfying
\begin{equation}\label{Hypfintro}
f^{[2]}(s,t,u) = \left\langle a(s,t), b(t,u) \right\rangle,\quad \forall (s,t,u) \in \mathbb{R}^3.
\end{equation}
If, in addition, either $f'$ is bounded or $A$ is bounded,  we prove that \eqref{Kop} holds
with the derivative $\dfrac{d}{dt}f(A+tK)_{|t=0}$ evaluated in the $\mathcal{S}^2(\H)$-norm.
The factorization property (\ref{Hypfintro}) can be regarded as a natural analogue
of (\ref{HSF}) for the second order divided difference.
We show in Proposition \ref{Besovfacto} that any $f\in B^2_{\infty 1}(\mathbb{R})$ satisfies \eqref{Hypfintro}
and in Proposition \ref{Haagerup} that $f\in C^2(\R)$ with $f^{[2]}$ in the dual space $(\ell^1_{\mathbb{R}}
\overset{h}{\otimes} \ell^1_{\mathbb{R}} \overset{h}{\otimes} \ell^1_{\mathbb{R}})^*,$ where $\overset{h}{\otimes}$ is the Haagerup tensor product, also satisfies \eqref{Hypfintro}.

The key technical tool in derivation of our main results is a multiple operator integral $\Gamma^{A_1,A_2, \ldots, A_n}(\phi)$ defined on the Cartesian product $\Sc^2(\H)\times\ldots\times\Sc^2(\H)$ for every bounded measurable symbol $\phi$. New algebraic and analytic properties of $\Gamma^{A_1,A_2, \ldots, A_n}(\phi)$ essential for the main results are obtained in Sections $\ref{MOI}$ and $\ref{Sectionapprox}$. To prove existence and continuity of higher order derivatives $\varphi^{(n)}$ in the $\Sc^2(\H)$-norm in Section $\ref{Perturbation}$, we take an approach of approximating symbols of multiple operator integrals and applying $w^*$-continuity of $\Gamma^{A_1,A_2, \ldots, A_n} : L^{\infty}\left(\prod_{i=1}^n \lambda_{A_i}\right) \longrightarrow
B_{n-1}(\mathcal{S}^2(\mathcal{H}) \times \mathcal{S}^2(\mathcal{H})\times \cdots \times \mathcal{S}^2(\mathcal{H}), \mathcal{S}^2(\mathcal{H}))$ stated in Proposition \ref{nintegrals} and the subsequent paragraph, where $\lambda_{A_i}$ is a scalar-valued spectral measure of $A_i$. This is different from the approaches of \cite{DeP-Suk,KPSS} resting on approximation of operators $A_1,A_2,\ldots,A_n$ used to prove existence of $\varphi'$  and from the approaches \cite{ACDS,Peller2006}
used to prove existence of $\varphi^{(n)}$ for smaller sets of functions $f$. To establish the trace formula \eqref{generalTF} and its particular case \eqref{Kop} in Section $\ref{traceform}$, we need $\varphi''$ to attain its values in $\Sc^1(\H)$ and the function $\tr(\varphi''(t))$ to be continuous on $[0,1]$. This imposes the Hilbert space factorization restriction \eqref{Hypfintro} on the functions $f$ for which we prove \eqref{generalTF}.

We use the following notations. For any $m\in\N$, we let $C(\mathbb{R}^m)$ be the vector space of all continuous functions from $\mathbb{R}^m$ into $\mathbb{C}$, we let $C_b(\mathbb{R}^m)$ be the subspace of all bounded continuous functions, and we let $C_0(\mathbb{R}^m)$ be the subspace of all continuous functions vanishing at infinity. Further, for any $p\in\N$ we let $C^p(\mathbb{R}^m)$ be the vector space of all $p$-times differentiable functions from $\mathbb{R}^m$ into $\mathbb{C}$. For $\mathcal{A}=\mathbb{R}$ or $\mathbb{C}$, we denote by $\text{Bor}(\mathcal{A})$ the space of bounded Borel functions from $\mathcal{A}$ into $\mathbb{C}$.
We denote the space of $p$-times differentiable functions $\varphi : \mathbb{R} \rightarrow \mathcal{S}^2(\mathcal{H})$
with continuous $p$th derivative $\varphi^{(p)} : \mathbb{R} \rightarrow \mathcal{S}^2(\mathcal{H})$ by $C^p(\mathbb{R},\Sc^2(\H))$.

\section{Multiple operator integrals}\label{MOI}

In this section, we recall a multiple operator integral introduced in \cite{Pav} and developed in \cite{CLS} and derive its key algebraic properties that underline our main results. We note that there are several other different constructions of multiple operator integrals \cite{ACDS,BS1,DK1956,Peller2006,PSS-SSF}, but they do not suit the generality of this paper because they are applicable to smaller sets of symbols. Comparison of different approaches to multiple operator integrals will be done in a special publication dedicated to the history of the subject.

Let $A$ be a (possibly unbounded) normal operator densely defined in $\mathcal H$ with spectrum $\sigma(A)$. Let $\lambda_A$ be a scalar-valued spectral measure of $A$. It means that $\lambda_A$ is a positive finite measure on the Borel subsets of $\sigma(A)$ such that $\lambda_A$ and $E^A$, the spectral measure of $A$, have the same sets of measure zero. We refer to \cite[Section 15]{Conway} and \cite[Section 2.1]{CLS} for the existence and the construction of such measure.

The Borel functional calculus for $A$ takes any bounded Borel
function $f : \sigma(A) \rightarrow \mathbb{C}$ to the bounded operator
$$
f(A):=\int_{\sigma(A)} f(t) \ \text{d}E^A(t)\,.
$$
This operator only depends on the class of $f$ in $L^{\infty}(\lambda_A)$.

Let $n\in\N, n\geq 2$ and let $A_1, A_2, \ldots, A_n$ be normal operators densely defined in $\mathcal{H}$ with scalar-valued spectral measures $\lambda_{A_1}, \ldots, \lambda_{A_n}$. We let $B_{n-1}(\mathcal{S}^2(\mathcal{H}) \times \mathcal{S}^2(\mathcal{H})
\times \cdots \times \mathcal{S}^2(\mathcal{H}), \mathcal{S}^2(\mathcal{H}))$ be the space of bounded $(n-1)$-linear mappings defined on the product of $n-1$ copies of $\mathcal{S}^2(\mathcal{H})$ valued in $\mathcal{S}^2(\mathcal{H})$. This is a dual space, and a predual is given by
$$\mathcal{S}^2(\mathcal{H}) \overset{\wedge}{\otimes} \cdots \overset{\wedge}{\otimes} \mathcal{S}^2(\mathcal{H}),$$
the projective tensor product of $n$ copies of $\mathcal{S}^2(\mathcal{H})$. See \cite[Section 3.1]{CLS} for more information in the case $n=3$.

Define a linear mapping
\begin{equation*}
\Gamma^{A_1,A_2, \ldots, A_n} : L^{\infty}(\lambda_{A_1}) \otimes \cdots \otimes L^{\infty}(\lambda_{A_n}) \rightarrow B_{n-1}(\mathcal{S}^2(\mathcal{H}) \times \mathcal{S}^2(\mathcal{H})
\times \cdots \times \mathcal{S}^2(\mathcal{H}), \mathcal{S}^2(\mathcal{H}))
\end{equation*}
such that for any $f_i \in L^{\infty}(\lambda_{A_i}), i=1, \ldots, n$ and for any $X_1, \ldots, X_{n-1} \in \mathcal{S}^2(\mathcal{H})$,
\begin{align}\label{MOItensor}
\left[\Gamma^{A_1,A_2, \ldots, A_n}(f_1\otimes\cdots\otimes f_n)\right]
& (X_1,\ldots, X_{n-1})\\ \nonumber
& =f_1(A_1)X_1f_2(A_2) \cdots f_{n-1}(A_{n-1})X_{n-1}f_n(A_n).
\end{align}

The following is proved in \cite[Theorem 5 and Proposition 6]{CLS}.

\begin{proposition}\label{nintegrals}
$\Gamma^{A_1,A_2, \ldots, A_n}$ extends to a unique $w^*$-continuous contraction
$$
\Gamma^{A_1,A_2, \ldots, A_n} : L^{\infty}\left(\prod_{i=1}^n
\lambda_{A_i}\right) \longrightarrow
B_{n-1}(\mathcal{S}^2(\mathcal{H}) \times \mathcal{S}^2(\mathcal{H})
\times \cdots \times \mathcal{S}^2(\mathcal{H}), \mathcal{S}^2(\mathcal{H})).
$$
\end{proposition}

The $w^*$-continuity of $\Gamma^{A_1,A_2, \ldots, A_n}$ means that if a net $(\phi_i)_{i\in I}$ in $L^{\infty}\left(\prod_{i=1}^n
\lambda_{A_i}\right)$ converges to $\phi \in L^{\infty}\left(\prod_{i=1}^n \lambda_{A_i}\right)$ in the $w^*$-topology, then for any $X_1, \ldots, X_{n-1} \in \mathcal{S}^2(\mathcal{H})$, the net
$$
\bigl(\left[\Gamma^{A_1,A_2, \ldots, A_n}(\phi_i)\right](X_1,\ldots, X_{n-1})\bigr)_{i\in I}
$$
converges to $\left[\Gamma^{A_1,A_2, \ldots, A_n}(\phi)\right](X_1,\ldots, X_{n-1})$ weakly in $\mathcal{S}^2(\mathcal{H})$.

\begin{definition}
For $\phi \in L^{\infty}\left(\prod_{i=1}^n\lambda_{A_i}\right)$, the transformation $\Gamma^{A_1,A_2, \ldots, A_n}(\phi)$ defined in \eqref{MOItensor} and extended in Proposition \ref{nintegrals} is called a multiple operator integral associated to $A_1, A_2, \ldots, A_n$ and $\phi$.
\end{definition}

If $n=2$, then the transformation $\Gamma^{A_1,A_2}(\phi)$ coincides with the double operator integral defined in \cite{BS1,BS2,BS3}.

In the case $n=3$, the description of the elements $\phi$ such that
$\Gamma^{A_1,A_2, A_3}(\phi)$ maps $\mathcal{S}^2(\mathcal{H}) \times \mathcal{S}^2(\mathcal{H})$ into $\mathcal{S}^1(\mathcal{H})$ was settled in \cite[Theorem 23]{CLS}. We recall this result for a future use.

\begin{theorem}\label{mainS1}
Let $A,B$ and $C$ be normal operators densely defined in $\mathcal{H}$
and let $\phi\in L^{\infty}(\lambda_A \times
\lambda_B \times \lambda_C)$. The following are equivalent:
\begin{enumerate}
\item[(i)] $\Gamma^{A,B,C}(\phi) \in B_2(\mathcal{S}^2(\mathcal{H})
\times \mathcal{S}^2(\mathcal{H}), \mathcal{S}^1(\mathcal{H})).$
\item[(ii)] There exist a separable Hilbert space $H$ and two functions
$$
a\in L^{\infty}(\lambda_A \times \lambda_B ; H) \qquad
\text{and} \qquad
b\in L^{\infty}(\lambda_B\times \lambda_C ; H)
$$
such that
$$
\phi(t_1,t_2,t_3)= \left\langle a(t_1,t_2),b(t_2,t_3) \right\rangle
$$
for a.e. $(t_1,t_2,t_3) \in \sigma(A) \times \sigma(B) \times \sigma(C).$
\end{enumerate}
In this case,
\begin{equation}\label{AllEqual}
\|\Gamma^{A,B,C}(\phi)\colon \mathcal{S}^2(\H)\times \mathcal{S}^2(\H)\longrightarrow \mathcal{S}^1(\H) \|=\inf
\bigl\{\|a\|_\infty \|b \|_\infty\bigr\},
\end{equation}
where the infimum runs over all pairs $(a,b)$ satisfying (ii).
\end{theorem}

In the sequel, we will work with densely defined selfadjoint operators $A_1, A_2, \ldots, A_n, n\in \mathbb{N}, n\geq 2$. If $\psi : \mathbb{R}^n \rightarrow \mathbb{C}$ is a bounded Borel function, let $\tilde{\psi}$ be the class of the restriction $\psi_{|\sigma(A_1) \times \sigma(A_2) \times \cdots \times \sigma(A_n)}$ in $L^{\infty}\left(\prod_{i=1}^n
\lambda_{A_i}\right)$. Then, we will denote by $\Gamma^{A_1,A_2, \ldots, A_n}(\psi)$ the multiple operator integral $\Gamma^{A_1,A_2, \ldots, A_n}(\tilde{\psi})$.

We deduce from Theorem \ref{mainS1} that if $f : \mathbb{R} \rightarrow \mathbb{C}$ is a $C^2$-function with bounded $f''$ such that $f^{[2]}$ satisfies the property \eqref{Hypfintro}, then for any selfadjoint operators $A,B,C$ on a separable Hilbert space $\mathcal{H}$, the class of $f^{[2]}_{|\sigma(A) \times \sigma(B) \times \sigma(C)}$ belongs to $L^{\infty}(\lambda_A \times
\lambda_B \times \lambda_C)$ and the triple operator integral $\Gamma^{A,B,C}(f^{[2]})$ maps $\mathcal{S}^2(\mathcal{H}) \times \mathcal{S}^2(\mathcal{H})$ into $\mathcal{S}^1(\mathcal{H})$. Hence, for any $X,Y \in \mathcal{S}^2(\mathcal{H})$, the trace
$$
\tr\left(\left[\Gamma^{A,B,C}(f^{[2]})\right](X,Y)\right)
$$
is a well defined element of $\mathbb{C}$.

For the rest of this section, we fix a bounded Borel function $\phi : \mathbb{R}^3 \rightarrow \mathbb{C}$ satisfying the factorization property \eqref{Hypfintro}, that is, we assume that there exist a separable Hilbert space $H$ and two bounded Borel functions $a, b : \mathbb{R}^2 \rightarrow H$ such that
\begin{equation}
\label{star}
\phi(s,t,u) = \left\langle a(s,t), b(t,u) \right\rangle,\quad \forall (s,t,u) \in \mathbb{R}^3.
\end{equation}
Let $(\epsilon_n)_n$ be a Hilbertian basis of $H$. For every $n\in\N$, let $a_n, b_n : \mathbb{R}^2 \rightarrow \mathbb{C}$ be the bounded Borel functions defined by
\begin{equation}\label{anbn}
a_n = \left\langle a, \epsilon_n \right\rangle \ \ \text{and} \ \ b_n = \left\langle \epsilon_n, b  \right\rangle.
\end{equation}
Thus,
\[\phi(s,t,u)=\sum_{n=1}^\infty a_n(s,t)\,b_n(s,t).\]

The following decomposition is a consequence of the proof of \cite[Theorem 23]{CLS}.

\begin{proposition}\label{new}
Let $\mathcal{H}$ be a separable Hilbert space and let $A_1,A_2,A_3$ be selfadjoint operators densely defined in $\mathcal{H}$.
Then,  for $\phi$ satisfying \eqref{star},
\begin{equation}
\label{CVseries}
\left[\Gamma^{A_1,A_2,A_3}(\phi)\right](X,Y) = \sum_{n=1}^{\infty} \Gamma^{A_1,A_2}(a_n)(X) \Gamma^{A_2,A_3}(b_n)(Y)
\end{equation}
for any $X,Y \in \mathcal{S}^2(\mathcal{H})$, where the series is absolutely convergent in $\mathcal{S}^1(\mathcal{H})$ and $a_n,b_n$ are defined in \eqref{anbn}. Moreover,
\begin{equation}\label{normineqS1}
\| \left[\Gamma^{A_1,A_2,A_3}(\phi)\right](X,Y) \|_1 \leq \|a\|_{\infty} \|b\|_{\infty} \|X\|_2 \|Y\|_2,
\end{equation}
where
$$\|a\|_{\infty} = \underset{(s,t)\in \mathbb{R}^2}{\sup}\ \|a(s,t)\|_H \qquad\hbox{and}\qquad
\|b\|_{\infty} = \underset{(t,u)\in \mathbb{R}^2}{\sup} \ \|b(t,u)\|_H.$$
\end{proposition}

The following representation for the trace of a triple operator integral is a crucial tool in derivation of our second main result.

\begin{proposition}\label{Simplf}
For $\phi$ satisfying \eqref{star}, let $\psi : \mathbb{R}^2 \rightarrow \mathbb{C}$ be defined by $$\psi(s,t)=\phi(s,t,s).$$ Let $A$ and $B$ be selfadjoint operators densely defined in $\mathcal{H}$. Then, for any $X,Y \in \mathcal{S}^2(\mathcal{H})$,
\begin{align*}
\tr\left(\big[\Gamma^{A,B,A}(\phi)\big](X,Y)\right)
=\tr\left(\big[\Gamma^{A,B}(\psi)\big](X)Y\right).
\end{align*}
\end{proposition}

To prove Proposition \ref{Simplf}, we need the following two lemmas.

\begin{lemma}\label{auxiliary1}
Let $A,B$ be selfadjoint operators densely defined in $\mathcal{H}$ and let $u,v : \mathbb{R}^2 \rightarrow \mathbb{C}$ be two bounded Borel functions. Then, for any $X,Y \in \mathcal{S}^2(\mathcal{H})$,
\begin{equation}\label{auxiliary}
\tr (\Gamma^{A,B}(u)(X) \Gamma^{B,A}(v)(Y)) = \tr \left(\left[\Gamma^{A,B}(u\tilde{v})\right](X)Y \right)
\end{equation}
where $\tilde{v}(s,t)=v(t,s)$.
\end{lemma}

\begin{proof}
Assume first that $u=u_1\otimes u_2$ and $v = v_1 \otimes v_2$ where $u_1,u_2,v_1,v_2 \in \text{Bor}(\mathbb{R})$. In this case, $\tilde{v} = v_2 \otimes v_1$ and $u\tilde{v} = u_1v_2 \otimes u_2v_1$. We have
$$\left[\Gamma^{A,B}(u)\right](X)=u_1(A)Xu_2(B) \ \ \text{and} \ \ \left[\Gamma^{B,A}(v)\right](Y) = v_1(B) Y v_2(A).$$
Then
\begin{align*}
\tr(\Gamma^{A,B}(u)(X) \Gamma^{B,A}(v)(Y))
& = \tr(u_1(A)Xu_2(B)v_1(B) Y v_2(A)) \\
& = \tr(v_2(A)u_1(A)Xu_2(B)v_1(B) Y) \\
& = \tr \left(\left[\Gamma^{A,B}(u\tilde{v})\right](X) Y \right).
\end{align*}
Hence, the equality $(\ref{auxiliary})$ is proved in this particular case and, by linearity, it holds true whenever $u,v \in \text{Bor}(\mathbb{R}) \otimes \text{Bor}(\mathbb{R})$.

In the general case, there exist two nets $(u_j)_j$ and $(v_i)_i$ in $\text{Bor}(\mathbb{R}) \otimes \text{Bor}(\mathbb{R})$ converging to $u$ and $v$ for the $w^*$-topologies of $L^{\infty}(\lambda_A \times \lambda_B)$ and $L^{\infty}(\lambda_B \times \lambda_A)$. The previous part implies that for all $i$ and $j$,
$$\tr (\Gamma^{A,B}(u_j)(X) \Gamma^{B,A}(v_i)(Y)) = \tr \left(\left[\Gamma^{A,B}(u_j\tilde{v_i})\right](X)Y \right).$$
Fix $j$. Since $(u_j \tilde{v_i})_i$ $w^*$-converges to $u_j\tilde{v}$, the $w^*$-continuity of $\Gamma^{A,B}$ implies that
$$\lim_i \tr \left(\left[\Gamma^{A,B}(u_j\tilde{v_i})\right](X)Y \right) = \tr \left(\left[\Gamma^{A,B}(u_j\tilde{v})\right](X)Y \right).$$
Similarly, the $w^*$-convergence of $(v_i)_i$ to $v$ implies that
$$\lim_i \tr (\Gamma^{A,B}(u_j)(X) \Gamma^{B,A}(v_i)(Y)) = \tr (\Gamma^{A,B}(u_j)(X) \Gamma^{B,A}(v)(Y)).$$
Hence, for all $j$,
$$\tr \left(\left[\Gamma^{A,B}(u_j\tilde{v})\right](X)Y \right) = \tr (\Gamma^{A,B}(u_j)(X) \Gamma^{B,A}(v)(Y)).$$
Taking the limit on $j$ in this equality and using $w^*$-continuity as above gives
$$\tr \left(\left[\Gamma^{A,B}(u\tilde{v})\right](X)Y \right) = \tr (\Gamma^{A,B}(u)(X) \Gamma^{B,A}(v)(Y)),$$
which is the desired equality.
\end{proof}

\begin{lemma}\label{weakCV}
For $\phi$ satisfying \eqref{star}, let $\psi : \mathbb{R}^2 \rightarrow \mathbb{C}$ be defined by $$\psi(s,t)=\phi(s,t,s).$$
Let $A,B$ be selfadjoint operators densely defined in $\mathcal{H}$. Then, for any $X\in \mathcal{S}^2(\mathcal{H})$, the series
\begin{equation}
\sum_{n=1}^{\infty} \left[\Gamma^{A,B}\left(a_n \widetilde{b_n}\right)\right](X)
\end{equation}
converges weakly to $\left[\Gamma^{A,B}(\psi)\right](X)$ in $\mathcal{S}^2(\mathcal{H})$, where $\widetilde{b_n}(t,u) = b_n(u,t).$
\end{lemma}

\begin{proof}
We set
$$
\psi_n = \sum_{k=1}^n a_k \widetilde{b_k}.
$$
Then we have
$$
\left[\Gamma^{A,B}(\psi_n)\right](X) = \sum_{k=1}^n \left[\Gamma^{A,B}(a_k \widetilde{b_k})\right](X).
$$
By the proof of \cite[Theorem 23]{CLS},
$$
w^*-\underset{n\rightarrow +\infty}{\lim}\psi_n = \psi
$$
in the $w^*$-topology of $L^{\infty}(\lambda_{A} \times \lambda_B)$. By Proposition $\ref{nintegrals}$, $\Gamma^{A,B}$ is $w^*$-continuous, so we have
$$
\left[\Gamma^{A,B}(\psi_n) \right](X) \underset{n\rightarrow +\infty}{\longrightarrow}
\left[\Gamma^{A,B}(\psi)\right](X)
$$
weakly in $\mathcal{S}^2(\mathcal{H})$.
\end{proof}

\begin{proof}[Proof of Proposition \ref{Simplf}] We have, by $(\ref{CVseries})$,
\begin{equation*}
\left[\Gamma^{A,B,A}(\phi)\right](X,Y) = \sum_{n=1}^{\infty} \Gamma^{A,B}(a_n)(X) \Gamma^{B,A}(b_n)(Y).
\end{equation*}
By Lemma $\ref{weakCV}$,
\begin{equation}\label{weakCV1}
\left[\Gamma^{A,B}(\psi)\right](X) = \sum_{n=1}^{\infty} \left[\Gamma^{A,B}\left(a_n \widetilde{b_n}\right)\right](X),
\end{equation}
where the series converges for the weak topology of $\mathcal{S}^2(\mathcal{H})$.\\
By continuity of $\tr$ on $\mathcal{S}^1(\mathcal{H})$, we get
$$\tr \left( \left[\Gamma^{A,B,A}(\phi) \right](X,Y) \right) = \sum_{n=1}^{\infty} \tr \left(\Gamma^{A,B}(a_n)(X)\Gamma^{B,A}(b_n)(Y)\right).$$
By Lemma $\ref{auxiliary1}$ and \eqref{weakCV1}, we obtain
\begin{align*}
\tr \left( \left[\Gamma^{A,B,A}(\phi) \right](X,Y) \right)
& = \sum_{n=1}^{\infty} \tr \left(\left[\Gamma^{A,B}\left(a_n \widetilde{b_n}\right)\right](X)Y\right) \\
& = \tr \left( \left(\sum_{n=1}^{\infty} \left[\Gamma^{A,B}\left(a_n \widetilde{b_n}\right)\right](X) \right) Y \right) \\
& = \tr \left( \left[\Gamma^{A,B}(\psi)\right](X)Y \right).
\end{align*}
\end{proof}

\section{Approximation in multiple operator integrals}\label{Sectionapprox}

In this section, we establish an approximation property for multiple operator integrals that will play a key role in the proofs of our main results.

Let $A$ be a selfadjoint operator densely defined in $\mathcal{H}$. We say that a sequence $(A_j)_j$ of selfadjoint operators is resolvent strongly convergent to $A$ if for any $z\in\mathbb{C}\setminus\mathbb{R}$, $(z-A_j)^{-1}\to (z-A)^{-1}$ in the strong operator topology (SOT). According to \cite[Theorem 8.20]{SimRe}, this is equivalent to
\begin{equation}\label{R-SOT}
\forall\, f\in C_b(\mathbb{R}),\qquad f(A_j)\overset{SOT}
{\longrightarrow} f(A)\qquad \hbox{when}\ j\to\infty.
\end{equation}

\begin{proposition}\label{Approx1}
Let  $n\in\N$, $n\geq 2$, $A_1, \ldots, A_n$ be selfadjoint operators densely defined in $\mathcal{H}$ and for all $1\leq i\leq n$, let $(A_i^j)_{j\in \mathbb{N}}$ be a sequence of selfadjoint operators densely defined in $\mathcal{H}$ resolvent strongly convergent to $A_i$.
Then for any $\phi\in C_b(\mathbb{R}^n)$ and for any $K_1, \ldots, K_{n-1}\in \mathcal{S}^2(\mathcal{H})$,
\begin{equation}\label{Approx2}
\underset{j \to +\infty}{\lim} \left\|\Gamma^{A_1^j,\ldots, A_n^j}(\phi)(K_1, \ldots, K_{n-1}) -\Gamma^{A_1, \ldots, A_n}(\phi)(K_1, \ldots, K_{n-1}) \right\|_2=0.
\end{equation}
\end{proposition}

\begin{proof}
For simplicity we write $\Gamma=\Gamma^{A_1, \ldots, A_n}$ and $\Gamma_j=\Gamma^{A_1^j, \ldots, A_n^j}$ along this proof. Since the space of finite rank operators is dense in $\mathcal{S}^2(\mathcal{H})$
and $\|\Gamma_j\| \leq 1$ for any $j\geq 1$, it suffices  to prove (\ref{Approx2}) in the case when $K_1, \ldots, K_{n-1}$ are finite rank operators. By linearity, we can further assume they are rank one operators. Thus from now on we assume that for all $1\leq i\leq n-1$,
$$
K_i=\overline{h_i}\otimes h_i'
$$
with $h_i, h_i' \in \mathcal{H}$, where for all $1 \leq i \leq n-1$, $\overline{h_i}\otimes h_i'$ is the operator defined for all $h\in \mathcal{H}$ by
$$(\overline{h_i}\otimes h_i')h = \left\langle h, h_i \right\rangle h_i'.$$
Assume first that $\phi=u_1\otimes \cdots \otimes u_n$, with $u_i \in C_b(\mathbb{R})$ for all $i$. In this case,
\begin{align*}
\Gamma_j(\phi)(K_1, \ldots, K_{n-1}) & = u_1(A_1^j)(\overline{h_1}\otimes h_1')\ldots
(\overline{h_{n-1}}\otimes h_{n-1}')u_n(A_n^j)\\
& = \left(\prod_{k=2}^{n-1} \left\langle u_k(A_k^j)h_k', h_{k-1} \right\rangle \right)
\overline{\overline{u_n}(A_n^j)(h_{n-1})}\otimes u_1(A_1^j)(h_1').
\end{align*}
By the assumption and (\ref{R-SOT}), this converges to
$$
\left(\prod_{k=2}^{n-1} \left\langle u_k(A_k)h_k', h_{k-1} \right\rangle \right)
\overline{\overline{u_n}(A_n)(h_{n-1})}\otimes u_1(A_1)(h_1'),
$$
which in turn is equal to
$\Gamma(\phi)(K_1, \ldots, K_{n-1})$. This shows (\ref{Approx2}) in this special case.
By linearity and standard approximation, this implies that
(\ref{Approx2}) holds true whenever $\phi$ belongs to the uniform
closure of $C_b(\mathbb{R})\otimes \cdots \otimes  C_b(\mathbb{R})$. In particular,
(\ref{Approx2}) holds true when $\phi\in C_0(\mathbb{R}^n)$.

The rest of the proof consists in reducing to this case by a more subtle (i.e.
non uniform) approximation process. Let
$(g_k)_{k\geq 1}$ be a sequence of functions in
$C_0(\mathbb{R})$ satisfying the following two properties:
$$
\forall\, k\in\N, \quad 0\leq g_k\leq 1,
\qquad\hbox{and}\qquad\forall\, r\in \mathbb{R},\quad
g_k(r)\overset{k\to\infty}{\longrightarrow} 1.
$$
These properties imply that for all $1 \leq i \leq n$, $g_k(A_i)\to I_{\small\mathcal{H}}$ strongly. Indeed let $h\in\mathcal{H}$, then, by the spectral theorem,
$$
\|g_k(A_i)h - h\|^2 = \int_{\sigma(A_i)}\bigl(1-g_k(r)\bigr)^2\,\text{d} \left<E^{A_i}(r)h,h\right>.
$$
Then by Lebesgue's dominated convergence theorem,
$\|g_k(A_i)h - h\|^2\to 0$ when $k\to\infty$.

We consider an arbitrary $\phi\in C_b(\mathbb{R}^n)$ and set
$$
\phi_k =(g_k\otimes g_k^2\otimes \cdots \otimes g_k^2 \otimes g_k)\phi,\qquad k\in\N.
$$
Clearly each $\phi_k$ belongs to $C_0(\mathbb{R}^n)$, hence satisfies (\ref{Approx2}).
A crucial observation is that for all $j,k\in\N$,
\begin{equation}\label{Duplicate1}
\Gamma_j(\phi_k)(K_1, \ldots, K_{n-1}) = \Gamma_j(\phi)\left(g_k(A_1^j)K_1g_k(A_2^j), \ldots,
g_k(A_{n-1}^j)K_{n-1}g_k(A_n^j)\right).
\end{equation}
The argument for this identity is essentially the same as the one for the proof of Lemma $\ref{auxiliary1}$. One first checks the validity of (\ref{Duplicate1}) in the case when $\phi$ belongs to $C_b(\mathbb{R})\otimes \cdots \otimes  C_b(\mathbb{R})$, then one uses the $w^*$-continuity of $\Gamma_j$ to obtain the general case. Details are left to the reader.
Likewise we have, for all $k\in\N$,
\begin{equation}\label{Duplicate2}
\Gamma(\phi_k)(K_1, \ldots, K_{n-1}) = \Gamma(\phi)\left(g_k(A_1)K_1g_k(A_2), \ldots,
g_k(A_{n-1})K_{n-1}g_k(A_n)\right).
\end{equation}

For any $k\in\N$ and any $1\leq i\leq n-1$,
$$
g_k(A_i)K_ig_k(A_{i+1}) = g_k(A_i)(\overline{h_i}\otimes h_i')g_k(A_{i+1}) =
\overline{g_k(A_{i+1})(h_i)}\otimes g_k(A_i)(h_i'),
$$
hence $g_k(A_i)K_ig_k(A_{i+1})\to K_i$ in $\mathcal{S}^2(\mathcal{H})$ when $k\to\infty$.

Let $\varepsilon>0$.  According to the above observation, we fix $k_0\in\N$
such that for any $1 \leq i \leq n-1$,
$$\| g_{k_0}(A_i) K_i g_{k_0}(A_{i+1}) - K_i\|_2 \leq \varepsilon.$$
Hence, there exists a constant $\alpha > 0$ depending only on $\| \phi\|_{\infty}, \|K_1\|_2, \ldots, \|K_{n-1}\|_2$ such that
$$
\left\|\Gamma(\phi_{k_0})(K_1, \ldots, K_{n-1}) - \Gamma(\phi)(K_1,\ldots, K_{n-1})\right\|_2
\leq \alpha \varepsilon.
$$

Now, using that for any $1 \leq i \leq n-1$, $g_{k_0}(A_i^j)K_i g_{k_0}(A_{i+1}^j) = \overline{g_{k_0}(A_{i+1}^j)(h_i)} \otimes g_{k_0}(A_i^j)(h_i')$ and the fact that $g_{k_0}(A_i^j) \to g_{k_0}(A_i)$ and $g_{k_0}(A_{i+1}^j) \to g_{k_0}(A_{i+1})$ strongly when $j\to \infty$, we see that $g_{k_0}(A_i^j)K_i g_{k_0}(A_{i+1}^j) \to g_{k_0}(A_i)K_i g_{k_0}(A_{i+1})$ in $\mathcal{S}^2(\mathcal{H})$ when $j\to \infty$. Hence, for a large enough $j_0 \in\N$, we have, for any $1 \leq i \leq n-1$,

$$\| g_{k_0}(A_i^j) K_i g_{k_0}(A_{i+1}^j) - K_i\|_2 \leq 2\varepsilon$$
for any $j\geq j_0$. We deduce that there exists a constant $\beta >0$ depending only on $\|\phi\|_\infty, \|K_1\|_2,$ $\ldots, \|K_{n-1}\|_2$ such that
$$
\forall\, j\geq j_0,
\qquad
\left\|\Gamma_j(\phi_{k_0})(K_1, \ldots, K_{n-1}) - \Gamma_j(\phi)(K_1, \ldots, K_{n-1})\right\|_2
\leq \beta \varepsilon.
$$
Now recall that $\phi_k$ satisfies (\ref{Approx2}). Hence changing $j_0$
into a bigger integer if necessary we also have
$$
\forall\, j\geq j_0,
\qquad
\left\|\Gamma_j(\phi_{k_0})(K_1, \ldots, K_{n-1}) - \Gamma(\phi_{k_0})(K_1, \ldots, K_{n-1})\right\|_2
\leq \varepsilon.
$$
We deduce from the above three estimates that
$$
\forall\, j\geq j_0,
\qquad
\left\|\Gamma_j(\phi)(K_1, \ldots, K_{n-1}) - \Gamma(\phi)(K_1, \ldots, K_{n-1})\right\|_2
\leq (\alpha + \beta +1) \varepsilon.
$$
This shows that $\phi$ satisfies (\ref{Approx2}).
\end{proof}

\begin{remark}
The result of Proposition 3.1 in the case when $n=2$ is established in \cite[Corollary 3.5]{DeP-Suk} under different assumptions on the sequences $(A_1^j)_{j\in\N}$ and $(A_2^j)_{j\in\N}$.
\end{remark}

We demonstrate below that every selfadjoint operator is a limit of bounded selfadjoint operators in the strong resolvent sense.

\begin{lemma}\label{existence}
Let $A$ be a selfadjoint operator densely defined in $\mathcal{H}$. Let $E$ be the spectral measure of $A$ and define $A_n:=E((-n,n))A$ for every $n\in\N$. Then, the sequence of bounded selfadjoint operators $(A_n)_{n=1}^\infty$ is resolvent strongly convergent to $A$.
\end{lemma}

\begin{proof}
Since $E((-n,n))$ converges to $I$ in the strong operator topology,
\begin{align}
\label{1An}
\lim_{n\rightarrow\infty}A_n g=Ag
\end{align}
for every $g\in D$, where $D$ is the domain of $A$.
Let $z\in\mathbb{C}\setminus\mathbb{R}$ and $f\in\mathcal{H}$. The mapping $A-z : D \rightarrow \mathcal{H}$ is a bijection so that $(A-z)^{-1}f \in D$. By standard properties of the resolvent,
\begin{align}
\label{2An}
(A_n-z)^{-1}f-(A-z)^{-1}f=(A_n-z)^{-1}(A-A_n)(A-z)^{-1}f.
\end{align}
The result follows from combining \eqref{1An} and \eqref{2An} and applying uniform boundedness of $(A_n-z)^{-1}$.
\end{proof}

We finish this section with a lemma that will be used in Section $\ref{Perturbation}$.

\begin{lemma}\label{RC}
Let $(A_n)_{n=1}^\infty$ be a sequence of selfadjoint operators converging resolvent strongly to a selfadjoint operator $A$. Let $K$ be a bounded self-adjoint operator. Then $(A_n+K)_{n=1}^\infty$ is resolvent strongly convergent to $A+K$.
\end{lemma}

\begin{proof}
Let $z\in\mathbb{C}$ be such that $\Im(z)\neq 0$.
Note that the operators
$$I-K(A+K-z)^{-1}$$
and
$$I-(A_n+K-z)^{-1}K$$
are bounded in the operator norm by $1+\|K\|/|\Im(z)|$. It is straightforward to see that
\begin{align}\label{rc3}
\begin{split}
& \left(I-(A_n+K-z)^{-1}K \right) \left[(A_n-z)^{-1}-(A-z)^{-1} \right] \left(I-K(A+K-z)^{-1} \right) \\
& = \left((A_n-z)^{-1}-(A_n+K-z)^{-1}K(A_n-z)^{-1} \right) \left(I-K(A+K-z)^{-1} \right)\\
& \ \ \ - \left(I-(A_n+K-z)^{-1}K \right) \left((A-z)^{-1}- (A-z)^{-1}K(A+K-z)^{-1} \right).
\end{split}
\end{align}
Since $K$ is bounded, $A$ and $A+K$ have the same domain, we have the resolvent formula
$$
(A-z)^{-1} - (A+K-z)^{-1} = (A-z)^{-1}K(A+K-z)^{-1}.
$$
Similarly, for every $n\geq 1$,
$$
(A_n-z)^{-1} - (A_n+K-z)^{-1} = (A_n+K-z)^{-1}K(A_n-z)^{-1}.
$$
Applying the latter in \eqref{rc3} gives
\begin{align*}
& \left(I-(A_n+K-z)^{-1}K \right) \left[(A_n-z)^{-1}-(A-z)^{-1} \right] \left(I-K(A+K-z)^{-1} \right) \\
& = (A_n+K-z)^{-1} \left(I-K(A+K-z)^{-1} \right) - \left(I-(A_n+K-z)^{-1}K \right) (A+K-z)^{-1}\\
& = (A_n+K-z)^{-1}-(A+K-z)^{-1}.
\end{align*}
Hence, for any $f\in\mathcal{H}$,
\begin{align*}
&\left\|\big((A_n+K-z)^{-1}-(A+K-z)^{-1}\big)f\right\|\\
&\leq\big(1+||K||/|\Im(z)|\big)
\left\|\big((A_n-z)^{-1}-(A-z)^{-1}\big)\big(I-K(A+K-z)^{-1}\big)f\right\|,
\end{align*}
which completes the proof of the lemma.
\end{proof}

\vspace{0.3cm}

\section{Differentiability of $t \mapsto f(A+tK) - f(A)$ in $\mathcal{S}^2(\mathcal{H})$}\label{Perturbation}

In this section we prove our first main result stated in the theorem below.

\begin{theorem}\label{Formula}
Let $A$ and $K$ be selfadjoint operators densely defined in $\mathcal{H}$ with $K\in \mathcal{S}^2(\mathcal{H})$. Let $n \in\N$ and $f\in C^n(\mathbb{R})$.  Assume that either $A$ is bounded or $f^{(i)}$ is bounded for all $1\leq i\leq n$. Consider the function
$$\varphi : t\in \mathbb{R} \mapsto f(A+tK) - f(A) \in \mathcal{S}^2(\mathcal{H}).$$
\begin{enumerate}
\item[$(i)$] The function $\varphi$ belongs to $C^n(\mathbb{R},\Sc^2(\H))$ and for every integer $1 \leq k \leq n$ and $t\in \mathbb{R}$,
\begin{equation}\label{diffenrential}
\dfrac{1}{k!} \varphi^{(k)}(t) = \left[ \Gamma^{A+tK, A+tK, \ldots, A+tK}(f^{[k]}) \right] (K, \ldots, K).
\end{equation}
\item[$(ii)$] The operator Taylor remainder satisfies
\begin{equation}\label{Taylor}
f(A+K)-f(A) - \sum_{k=1}^{n-1} \dfrac{1}{k!} \varphi^{(k)}(0) = \left[ \Gamma^{A+K, A, \ldots, A}(f^{[n]}) \right] (K, \ldots, K).
\end{equation}
\end{enumerate}
\end{theorem}

We first recall the definition of the divided difference. Let $f\in C^1(\mathbb{R})$. The divided
difference of the first order $f^{[1]}\colon\mathbb{R}^2\to\mathbb{C}$
is defined by
\begin{align*}
{f^{[1]} (x_0,x_1)} :=
\begin{cases}\frac
{ f(x_0) - f(x_1)}{x_0-x_1}, & \text{if}\ x_0
\neq x_1 \\
f'(x_0) & \text{if}\ x_0=x_1
\end{cases}, \qquad x_0, x_1\in\mathbb{R}.
\end{align*}
The function $f^{[1]}$ belongs to $C(\mathbb{R}^2)$ and
if $f'$ is bounded, then $f^{[1]}\in C_b(\mathbb{R}^2)$.
If $n\geq 2$ and $f\in C^n(\mathbb{R})$, the divided difference of the $n$th
order $f^{[n]}\colon\mathbb{R}^{n+1}\to\mathbb{C}$
is defined recursively by
\begin{align*}
{f^{[n]} (x_0,x_1,\ldots,x_n)} :=
\begin{cases}\frac
{f^{[n-1]}(x_0,x_2,\ldots,x_n) - f^{[n-1]}(x_1, x_2 \ldots,x_n)}{x_0-x_1}, & \text{if}\ x_0
\neq x_1 \\
\partial_1 f^{[n-1]}(x_1,x_2,\ldots,x_n) & \text{if}\ x_0=x_1
\end{cases},
\end{align*}
for all $x_0, \ldots, x_n \in \mathbb{R}$, where $\partial_i$ stands for the partial derivative with respect to the $i$-th variable.

It is well-known that $f^{[n]}$ is symmetric under permutation of its arguments. Therefore, for all $1 \leq i \leq n$ and for all $x_0, \ldots, x_n \in \mathbb{R}$,
\begin{align*}
{f^{[n]} (x_0,x_1,\ldots,x_n}) =
\begin{cases}\frac
{f^{[n-1]}(x_0,\ldots,x_{i-1}, x_{i+1}, \ldots,x_n) - f^{[n-1]}(x_0,\ldots,x_{i-2}, x_{i},x_{i+1},\ldots,x_n)}{x_{i-1}-x_i}, & \text{if}\ x_{i-1}
\neq x_i \\
\partial_i f^{[n-1]}(x_1,\ldots,x_n) & \text{if}\ x_{i-1}=x_i.
\end{cases},
\end{align*}
Note for further use that for all $1 \leq i \leq n$ and for all $(x_0, \ldots, x_n) \in\mathbb{R}^{n+1}$,
\begin{equation}\label{formuladivdiff}
{f^{[n]} (x_0, \ldots, x_n)}  = \int_0^1\partial_i f^{[n-1]}(x_0, \ldots, x_{i-2}, tx_{i-1} + (1-t)x_i,x_{i+1}, \ldots, x_n)\,\text{d}t\,.
\end{equation}
The function $f^{[n]}$ belongs to $C(\mathbb{R}^{n+1})$;
if $f^{(n)}$ is bounded, then $f^{[n]}\in C_b(\mathbb{R}^{n+1})$.\\

In order to prove Theorem $\ref{Formula}$, we need algebraic properties of multiple operator integrals established below.

\begin{lemma}\label{separation}
Let $n,k\in\N$, $n \geq 3$ and $1 \leq k \leq n-2$. Let $u\in C_b(\mathbb{R}^{k+1})$ and $v \in C_b(\mathbb{R}^{n-k})$.
Define
$$(uv)(t_1, \ldots, t_n) = u(t_1, \ldots, t_{k+1}) v(t_{k+1}, \ldots, t_n)\quad\text{for } (t_1, \ldots, t_n) \in \mathbb{R}^n$$
and let $A_1, \ldots, A_n$ be selfadjoint operators densely defined in a separable Hilbert space $\mathcal{H}$. Then, for any $K_1, \ldots, K_{n-1} \in \mathcal{S}^2(\mathcal{H})$,
\begin{align*}
& \Gamma^{A_1, \ldots, A_n}(uv)(K_1, \ldots, K_{n-1})\\
& \ \ \ \ \ \ \ \ = \Gamma^{A_1, \ldots, A_{k+1}}(u)(K_1, \ldots, K_k) \Gamma^{A_{k+1}, \ldots, A_n}(v)(K_{k+1}, \ldots, K_{n-1}).
\end{align*}
\end{lemma}

\begin{proof}
The proof is straightforward for $u$ and $v$ elementary tensors of elements of $C_b(\mathbb{R})$. Then, one uses the $w^*$-continuity of multiple operator integrals as in Lemma $\ref{auxiliary1}$ to obtain the general case.
\end{proof}

\begin{lemma}\label{simplification}
Let $n \geq 2$ be an integer. Let $A_1, \ldots, A_{n-1}, A, B\in\mathcal{B}(\mathcal{H})$ and assume that $B-A \in \mathcal{S}^2(\mathcal{H})$. Let $f\in C^n(\mathbb{R})$. Then, for all $K_1, \ldots, K_{n-1} \in \mathcal{S}^2(\mathcal{H})$ and for any $1 \leq i \leq n$ we have
\begin{align*}
& \ \ \ \ \ \ \ \ \ \ \ \ \ \left[ \Gamma^{A_1, \ldots A_{i-1}, B, A_i, \ldots, A_{n-1}}(f^{[n-1]})\right] (K_1, \ldots, K_{n-1}) \\
& \ \ \ \ \ \ \ \ \ \ \ \ \ - \left[ \Gamma^{A_1, \ldots A_{i-1}, A, A_i, \ldots, A_{n-1}}(f^{[n-1]})\right] (K_1, \ldots, K_{n-1}) \\
& = \left[ \Gamma^{A_1, \ldots, A_{i-1}, B, A, A_i, \ldots, A_{n-1}}(f^{[n]})\right](K_1, \ldots, K_{i-1}, B-A, K_i, \ldots K_{n-1}).
\end{align*}
\end{lemma}

\begin{proof}
It will be convenient to extend the definition of the divided difference as follows. Let $m\in \mathbb{N}$ and $1 \leq i \leq m$. For any $\phi\in C^1(\mathbb{R}^m)$, we define
a function $\phi^{[1]}_i\colon \mathbb{R}^{m+1} \to \mathbb{C}$ by
$$\phi^{[1]}_i(x_0, \ldots, x_m) = \int_0^1 \partial_i \phi(x_0, \ldots, x_{i-2}, t x_{i-1} + (1-t)x_i, x_{i+1}, \ldots, x_m)\,\text{d}t$$
for all $(x_0, \ldots, x_m) \in \mathbb{R}^{m+1}$. The index $i$ in the notation $\phi^{[1]}_i$ refers to the $i$-th variable derivation $\partial_i$. It follows from $(\ref{formuladivdiff})$ that for any $f\in C^n(\mathbb{R})$,
\begin{equation}\label{nextrank}
(f^{[n-1]})_i^{[1]} = f^{[n]}.
\end{equation}

\noindent For $\phi \in C(\mathbb{R}^n)$, write
$$\Gamma_A(\phi) =  \left[\Gamma^{A_1, \ldots A_{i-1}, A, A_i, \ldots, A_{n-1}}(\phi)\right](K_1, \ldots, K_{n-1})$$
and
$$\Gamma_B(\phi) =  \left[ \Gamma^{A_1, \ldots A_{i-1}, B, A_i, \ldots, A_{n-1}}(\phi)\right](K_1, \ldots, K_{n-1}).$$
For $\psi\in C(\mathbb{R}^{n+1})$, write
$$\Gamma_{B,A}(\psi) = \left[\Gamma^{A_1, \ldots, A_{i-1}, B, A, A_i, \ldots, A_{n-1}}(\psi)\right](K_1, \ldots, K_{i-1}, B-A, K_i, \ldots, K_{n-1}) .$$
We will show that for any $\phi \in C^1(\mathbb{R}^n)$,
\begin{equation}\label{intermediate}
\Gamma_B(\phi) - \Gamma_A(\phi) = \Gamma_{B,A}\left( \phi_i^{[1]} \right).
\end{equation}
Then the result follows by applying this formula to $\phi = f^{[n-1]}$, together with $(\ref{nextrank})$.

Assume first that $\phi=u_1\otimes \cdots \otimes u_n$ for functions $u_j \in C^1(\mathbb{R})$, i.e., $\phi(t_1,\ldots,t_n)=u_1(t_1)\ldots u_n(t_n)$ for every $(t_1,\ldots,t_n) \in \mathbb{R}^n$. Then
$$\partial_i \phi = u_1 \otimes \cdots \otimes u_{i-1} \otimes u_i' \otimes u_{i+1} \otimes \cdots \otimes u_n.$$
Hence,
$$\phi_i^{[1]} = u_1 \otimes \cdots \otimes u_{i-1} \otimes u_i^{[1]} \otimes u_{i+1} \otimes \cdots \otimes u_n.$$
By Lemma $\ref{separation}$ we have
\begin{align*}
& \Gamma_{B,A}(\phi_i^{[1]}) \\
& = \left[\Gamma^{A_1, \ldots, A_{i-1}, B}(u_1 \otimes \cdots \otimes u_{i-1} \otimes 1) \right](K_1, \ldots, K_{i-1}) \left[\Gamma^{B,A, A_i}(u_i^{[1]} \otimes 1)\right](B-A, K_i) \\
& \ \ \ \ \left[\Gamma^{A_i, \ldots, A_{n-1}}(u_{i+1} \otimes \cdots \otimes u_n)\right](K_{i+1}, \ldots,K_{n-1}).
\end{align*}
By \eqref{MOItensor} and the representation
\begin{equation}\label{f(A)-f(B)}
f(A+K)-f(A) = \bigl[\Gamma^{A+K,A}(f^{[1]})\bigr](K).
\end{equation} (established in \cite[Corollary to Theorem 4.5]{BS1}),

\begin{align*}
\left[\Gamma^{B,A, A_i}(u_i^{[1]} \otimes 1)\right](B-A, K_i)
& = \left[\Gamma^{B,A}(u_i^{[1]})\right](B-A) K_i \\
& = (u_i(B) - u_i(A))K_i.
\end{align*}
Hence,
\begin{align*}
\Gamma_{B,A}(\phi_i^{[1]})
& = u_1(A_1)K_1 \ldots u_{i-1}(A_{i-1}) K_{i-1} (u_i(B) - u_i(A))K_i  u_{i+1}(A_i)K_{i+1} \ldots u_n(A_{n-1}) \\
& = u_1(A_1)K_1 \ldots u_{i-1}(A_{i-1}) K_{i-1} u_i(B) K_i  u_{i+1}(A_i)K_{i+1} \ldots u_n(A_{n-1}) \\
& \ \ \ - u_1(A_1)K_1 \ldots u_{i-1}(A_{i-1}) K_{i-1} u_i(A) K_i  u_{i+1}(A_i)K_{i+1} \ldots u_n(A_{n-1}) \\
& = \Gamma_B(\phi) - \Gamma_A(\phi).
\end{align*}
This shows $(\ref{intermediate})$ in the case when $\phi=u_1\otimes \cdots \otimes u_n$. By linearity this immediately implies that $(\ref{intermediate})$ holds true whenever $\phi \in C^1(\mathbb{R}) \otimes \cdots \otimes C^1(\mathbb{R}).$ Note that this space contains the $n$-variable polynomial functions.

Now consider an arbitrary $\phi \in C^1(\mathbb{R}^n)$. Let $M > 0$ be a constant such that the spectra of $A_1, \ldots, A_{n-1}, A$ and $B$ are included in $[-M, M]$. By continuity of $\partial_i \phi$ there exists a sequence $(Q_m)_{m\geq 1}$ of $n$-variable polynomial functions such that $Q_m \rightarrow \partial_i \phi$ uniformly on $[-M,M]^n$. For any $m\geq 1$, we set
$$P_m(t_1, \ldots, t_n) = \int_0^{t_i} Q_m(t_1, \ldots, t_{i-1}, \theta, t_{i+1}, \ldots, t_n)\,\text{d}\theta$$
for all $(t_1, \ldots,t_n) \in \mathbb{R}^n$. This is also an $n$-variable polynomial function. Next we introduce $w(t_1, \ldots, t_{i-1}, t_{i+1}, \ldots, t_n) = \phi(t_1, \ldots, t_{i-1}, 0, t_{i+1}, \ldots, t_n)$. The function $w$ belongs to $C^1(\mathbb{R}^{n-1})$ and for any real numbers $t_1, \ldots, t_n$, we have
$$\phi(t_1, \ldots, t_n) = w(t_1, \ldots, t_{i-1}, t_{i+1}, \ldots, t_n) + \int_0^{t_i} \partial_i \phi (t_1, \ldots, t_{i-1}, \theta, t_{i+1}, \ldots, t_n)\,\text{d}\theta.$$
Hence,
\begin{align*}
& |\phi(t_1, \ldots, t_n) - w(t_1, \ldots, t_{i-1}, t_{i+1}, \ldots, t_n) - P_n(t_1, \ldots, t_n)| \\
& \leq \int_0^{t_i} |\partial_i \phi (t_1, \ldots, t_{i-1}, \theta, t_{i+1}, \ldots, t_n) - Q_m(t_1, \ldots, t_{i-1}, \theta, t_{i+1}, \ldots, t_n) | \,\text{d}\theta.
\end{align*}
Consequently, $P_m + w \rightarrow \phi$ uniformly on $[-M, M]^n$, where we naturally see $w$ as an element of $C^1(\mathbb{R}^n)$. Let $(w_m)_{m\in \mathbb{N}}$ be a sequence of $(n-1)$-variable polynomial functions converging uniformly to $w$ on $[-M,M]^{n-1}$ and see $w_m$ as an element of $C^1(\mathbb{R}^n)$. The latter implies that $P_m + w_m \rightarrow \phi$ uniformly on $[-M, M]^n$. By construction, $\partial_i P_m = Q_m$ and $\partial_i w_m = 0$ hence we also obtain that $(P_m + w_m)_i^{[1]} \rightarrow \phi_i^{[1]}$ uniformly on $[-M, M]^{n+1}$. Since $P_m + w_m$ belongs to $C^1(\mathbb{R}) \otimes \cdots \otimes C^1(\mathbb{R})$, it satisfies $(\ref{intermediate})$. The above approximation property and Proposition \ref{nintegrals} imply that $\phi$ satisfies $(\ref{intermediate})$ as well.
\end{proof}

The following corollary is an extension of Lemma $\ref{simplification}$ to unbounded operators.

\begin{corollary}\label{simplificationUNB}
Let $n \geq 2$ be an integer. Let $A_1, \ldots, A_{n-1}, A, K$ be selfadjoint operators densely defined in $\mathcal{H}$ and assume that $K \in \mathcal{S}^2(\mathcal{H})$. Let $f\in C^n(\mathbb{R})$ be such that $f^{(n-1)}$ and $f^{(n)}$ are bounded. Then, for all $K_1, \ldots, K_{n-1} \in \mathcal{S}^2(\mathcal{H})$ and for any $1 \leq i \leq n$ we have
\begin{align*}
& \ \ \ \ \ \ \ \ \ \ \ \ \ \left[ \Gamma^{A_1, \ldots A_{i-1}, A+K, A_i, \ldots, A_{n-1}}(f^{[n-1]})\right] (K_1, \ldots, K_{n-1}) \\
& \ \ \ \ \ \ \ \ \ \ \ \ \ - \left[ \Gamma^{A_1, \ldots A_{i-1}, A, A_i, \ldots, A_{n-1}}(f^{[n-1]})\right] (K_1, \ldots, K_{n-1}) \\
& = \left[ \Gamma^{A_1, \ldots, A_{i-1}, A + K, A, A_i, \ldots, A_{n-1}}(f^{[n]})\right](K_1, \ldots, K_{i-1}, K, K_i, \ldots K_{n-1}).
\end{align*}
\end{corollary}

\begin{proof}
For all $1\leq k \leq n-1$, let $(A_k^j)_{j\in \mathbb{N}}$ be a sequence of bounded selfadjoint operators on $\mathcal{H}$ converging resolvent strongly to $A_k$. Such sequence exists by Lemma $\ref{existence}$. Similarly, let $(A^j)_{j\in \mathbb{N}}$ be a sequence of bounded selfadjoint operators converging resolvent strongly to $A$. According to Lemma $\ref{simplification}$, we have, for all $j$,
\begin{align*}
& \ \ \ \ \ \ \ \ \ \ \ \ \ \left[ \Gamma^{A_1^j, \ldots A_{i-1}^j, A^j + K, A_i^j, \ldots, A_{n-1}^j}(f^{[n-1]})\right] (K_1, \ldots, K_{n-1}) \\
& \ \ \ \ \ \ \ \ \ \ \ \ \ - \left[ \Gamma^{A_1^j, \ldots A_{i-1}^j, A^j, A_i^j, \ldots, A_{n-1}^j}(f^{[n-1]})\right] (K_1, \ldots, K_{n-1}) \\
& = \left[ \Gamma^{A_1^j, \ldots, A_{i-1}^j, A^j + K, A^j, A_i^j, \ldots, A_{n-1}^j}(f^{[n]})\right](K_1, \ldots, K_{i-1}, K, K_i, \ldots K_{n-1}).
\end{align*}
By Lemma $\ref{RC}$, $A^j + K \to A + K$ resolvent strongly when $j \to \infty$. Moreover, the boundedness of $f^{(n-1)}$ and $f^{(n)}$ imply that of $f^{[n-1]}$ and $f^{[n]}$. We obtain the desired equality by passing to the limit in the above equality and applying Proposition $\ref{Approx1}$.
\end{proof}

The proof of the first main result is given below.

\begin{proof}[Proof of Theorem $\ref{Formula}$] Assume first that $A$ is bounded.

$(i)$ We prove the first claim by induction on $k$, $1 \leq k \leq n$. Let $k=1$ and $t\in \mathbb{R}$. We want to show that the limit
$$\underset{s \rightarrow 0}{\lim} \ \dfrac{\varphi(t+s) - \varphi(t)}{s}$$
exists in $\mathcal{S}^2(\mathcal{H})$ and equals $\left[ \Gamma^{A + tK, A + tK}(f^{[1]}) \right](K)$.
By $(\ref{f(A)-f(B)})$, we have
\begin{align*}
\dfrac{\varphi(t+s) - \varphi(t)}{s}
& = \dfrac{f(A + (t+s)K) - f(A + tK)}{s} \\
& = \left[\Gamma^{A+(t+s)K, A+tK}(f^{[1]})\right](K).
\end{align*}
By Lemma $\ref{RC}$, $A+(t+s)K \to A +tK$ resolvent strongly as $s \to 0$. By assumption $A$ and $K$ are bounded, so there exists a bounded interval $I \subset \mathbb{R}$ such that for $s$ small enough, $\sigma(A + (t+s)K) \subset I$. Since $f \in C^1(\mathbb{R})$, $f^{[1]}$ is continuous and, hence, bounded on $I \times I$. Let $F \in C_b(\mathbb{R}^2)$ be such that $F_{|I\times I}=f^{[1]}$. By Proposition $\ref{Approx1}$ applied to $F$ we get
$$\underset{s \to 0}{\lim} \ \left[\Gamma^{A+(t+s)K, A+tK}(f^{[1]})\right](K) = \left[\Gamma^{A + tK, A+tK}(f^{[1]})\right](K) \ \ \text{in} \ \ \mathcal{S}^2(\mathcal{H}).$$
This shows that $\varphi'(t) =  \left[ \Gamma^{A + tK, A + tK}(f^{[1]}) \right](K)$.

Since $A + tK \to A + t_0K$ in $\mathcal{B}(\mathcal{H})$ and, hence, in the strong resolvent sense as $t \to t_0$, we obtain $\varphi'(t) \to \varphi'(t_0)$ in $\mathcal{S}^2(\mathcal{H})$ as $t \to t_0$ by Proposition $\ref{Approx1}$. This confirms continuity of $\varphi'$.

Now let $1 \leq k \leq n-1$ and assume that $\varphi \in C^k(\mathbb{R})$ and for all $1 \leq j \leq k$ and $t\in \mathbb{R}$,
\begin{equation}
\varphi^{(j)}(t) = j! \left[ \Gamma^{A+tK, A+tK, \ldots, A+tK}(f^{[j]}) \right] (K, \ldots, K).
\end{equation}
We want to prove that $\varphi \in C^{k+1}(\mathbb{R})$ with a derivative of $(k+1)$-th order given by $(\ref{diffenrential})$. Let $s,t \in \mathbb{R}$. We have
\begin{align*}
& \dfrac{\varphi^{(k)}(t+s) - \varphi^{(k)}(t)}{s} \\
& = \dfrac{k!}{s} \left[ \Gamma^{A+(t+s)K,\ldots, A+(t+s)K}(f^{[k]}) - \Gamma^{A+tK,\ldots, A+tK}(f^{[k]}) \right](K, \ldots, K) \\
& = \dfrac{k!}{s} \sum_{i=1}^{k+1} \left[ \Gamma^{(A+tK)^{i-1}, (A+(t+s)K)^{k-i+2}}(f^{[k]}) - \Gamma^{(A+tK)^i, (A+(t+s)K)^{k-i+1}}(f^{[k]}) \right](K, \ldots, K)
\end{align*}
where for instance, $(A + tK)^i$ stands for $A+tK, \ldots, A+tK$ (i terms).
By Lemma $\ref{simplification}$, we have for all $1\leq i \leq k+1$,
\begin{align*}
& \dfrac{1}{s} \left[ \Gamma^{(A+tK)^{i-1}, (A+(t+s)K)^{k-i+2}}(f^{[k]}) - \Gamma^{(A+tK)^i, (A+(t+s)K)^{k-i+1}}(f^{[k]}) \right](K, \ldots, K) \\
& = \left[ \Gamma^{(A+tK)^{i-1}, A+(t+s)K, A+tK, (A+(t+s)K)^{k-i+1}}(f^{[k+1]}) \right](K, \ldots, K).
\end{align*}
Moreover, using strong resolvent convergence like in the first part of the proof, we can see that this term converges in $\mathcal{S}^2(\mathcal{H})$, as $s\rightarrow 0$, to
$$\left[ \Gamma^{A+tK,\ldots, A+tK}(f^{[k+1]}) \right](K, \ldots, K).$$
Hence,
\begin{align*}
\underset{s \rightarrow 0}{\lim} \ \dfrac{\varphi^{(k)}(t+s) - \varphi^{(k)}(t)}{s}
& = k! \sum_{i=1}^{k+1} \left[ \Gamma^{A+tK,\ldots, A+tK}(f^{[k+1]}) \right](K, \ldots, K) \\
& = (k+1)! \left[ \Gamma^{A+tK,\ldots, A+tK}(f^{[k+1]}) \right](K, \ldots, K).
\end{align*}
Finally, the continuity of $\varphi^{(k+1)}$ follows by the same argument as the continuity of $\varphi'$. This concludes the proof of $(i)$.

$(ii)$. We will prove the second claim by induction on $n$. The case $n=1$ follows from $(\ref{f(A)-f(B)})$. Now let $n\in \mathbb{N}$ and $f\in C^{n+1}(\mathbb{R})$. Assume that we have
$$f(A+K)-f(A) - \sum_{k=1}^{n-1} \dfrac{1}{k!} \varphi^{(k)}(0) = \left[ \Gamma^{A+K, A, \ldots, A}(f^{[n]}) \right] (K, \ldots, K).$$
We have
\begin{align*}
f(A+K)-f(A) - \sum_{k=1}^{n} \dfrac{1}{k!} \varphi^{(k)}(0)
& = f(A+K)-f(A) - \sum_{k=1}^{n-1} \dfrac{1}{k!} \varphi^{(k)}(0) - \dfrac{1}{n!} \varphi^{(n)}(0) \\
& = \left[ \Gamma^{A+K, A, \ldots, A}(f^{[n]}) \right] (K, \ldots, K) - \dfrac{1}{n!} \varphi^{(n)}(0).
\end{align*}
By $(i)$, we have
$$\dfrac{1}{n!} \varphi^{(n)}(0) = \left[ \Gamma^{A, A, \ldots, A}(f^{[n]}) \right] (K, \ldots, K).$$
Using Lemma $\ref{simplification}$, we obtain
$$f(A+K)-f(A) - \sum_{k=1}^{n} \dfrac{1}{k!} \varphi^{(k)}(0) = \left[ \Gamma^{A+K, A, \ldots, A}(f^{[n+1]}) \right] (K, \ldots, K)$$
which is the desired equality.

Assume now that $A$ is unbounded and that for all $1\leq i \leq n$, $f^{(i)}$ is bounded. Then, for all $1\leq i \leq n, f^{[i]}$ is bounded. Hence, applying Corollary $\ref{simplificationUNB}$ instead of Lemma $\ref{simplification}$ and following the same lines as in the proof of the bounded case, we obtain the unbounded case.
\end{proof}

Theorem $\ref{Formula}$, Proposition $\ref{Approx1}$ and Lemma $\ref{RC}$ have the following consequence.

\begin{corollary}\label{CoroTaylor}
Let $A$ be a selfadjoint operator densely defined in $\mathcal{H}$ and let $(A_j)_{j\in \mathbb{N}}$ be a sequence of selfadjoint operators in $\mathcal{B}(\mathcal{H})$ converging resolvent strongly to $A$. Let $n\in\N$ and let $f\in C^n(\mathbb{R})$ be such that $f^{(n)}$ is bounded. Let $K = K^* \in \mathcal{S}^2(\mathcal{H})$ and define, for every $j\in\N$,
$$\varphi_j : t\in \mathbb{R} \mapsto f(A_j +tK) - f(A_j) \in \mathcal{S}^2(\mathcal{H}).$$
Then, for every $t \in \mathbb{R}$,
$$\underset{j\to \infty}{\lim} \ \dfrac{\varphi_j^{(n)}(t)}{n!} = \left[\Gamma^{A+tK, \ldots, A+tK}(f^{[n]}) \right](K, \ldots, K)$$
and
$$\underset{j\to \infty}{\lim} \left( f(A_j+K)-f(A_j) - \sum_{k=1}^{n-1} \dfrac{1}{k!} \varphi_j^{(k)}(0) \right) = \left[ \Gamma^{A+K, A, \ldots, A}(f^{[n]}) \right] (K, \ldots, K),$$
where the limits are taken in $\mathcal{S}^2(\mathcal{H})$.
\end{corollary}

\section{An extension of the Koplienko trace formula}\label{traceform}

In this section, we prove our second main theorem and demonstrate two important classes of functions that satisfy this result.

\begin{theorem}\label{ThmA}
Let $f\in C^2(\R)$ be such that $f''$ is bounded. Assume that there exist a separable Hilbert space $H$ and two bounded Borel functions $a, b : \mathbb{R}^2 \rightarrow H$ such that
\begin{equation}\label{Hypf}
f^{[2]}(s,t,u) = \left\langle a(s,t), b(t,u) \right\rangle,\quad \forall (s,t,u) \in \mathbb{R}^3.
\end{equation}
Let $A,K$ be selfadjoint operators densely defined in $\mathcal{H}$ with $K \in \mathcal{S}^2(\mathcal{H})$. Let $\eta$ be the Koplienko spectral shift function associated to $(A,K)$. Then,
\begin{equation}\label{traceformint}
\tr\left(\left[\Gamma^{A+K,A,A}(f^{[2]})\right](K,K)\right) = \int_{\mathbb{R}} f''(t) \eta(t) \, dt.
\end{equation}
If, in addition, either $f'$ is bounded or $A$ is bounded, then
\begin{equation*}
\tr\Big(f(A+K) - f(A) - \dfrac{d}{dt}f(A+tK)_{|t=0}\Big) = \int_{\mathbb{R}} f''(t) \eta(t) \, dt,
\end{equation*}
where $\dfrac{d}{dt}f(A+tK)_{|t=0}\overset{\text{def}}{=}\mathcal{S}^2(\mathcal{H})\text{-}\lim\limits_{t\rightarrow 0^+}\frac{f(A+tK)-f(A)}{t}=\bigl[\Gamma^{A,A}(f^{[1]})\bigr](K)$.
\end{theorem}

For simplicity, we denote
\begin{equation*}
\Gamma(A,K) \overset{\text{def}}{=} \left[\Gamma^{A+K,A,A}(f^{[2]})\right](K,K).
\end{equation*}
By Theorem \ref{mainS1}, $\Gamma(A,K)\in\mathcal{S}^1(\mathcal{H})$ and by \eqref{normineqS1} of  Proposition \ref{new},
\begin{equation*}
\| \Gamma(A,K) \|_1 \leq \|a\|_{\infty} \|b\|_{\infty} \|K\|_2^2.
\end{equation*}
By Theorem \ref{Formula},
\begin{equation*}
\label{GammainS1}
\Gamma(A,K) = f(A+K) - f(A) - \dfrac{d}{dt}f(A+tK)_{|t=0}
\end{equation*}
if either $f'$ is bounded or $A$ is bounded.

The proof of Theorem \ref{ThmA} is obtained in a technically more subtle way than the proof of its
first order counterpart \eqref{KreinsFormula} in \cite{Peller2016}. Our goal is
achieved with help of the following lemmas.

\begin{lemma}\label{rir}
Let $A,K$ be selfadjoint operators densely defined in $\mathcal{H}$ with $K \in
 \mathcal{S}^2(\mathcal{H})$. Let $f\in C^2(\R)$ be such that $f^{[2]}$ admits a factorization \eqref{Hypf} for some separable Hilbert space $H$ and some bounded Borel functions $a,b : \mathbb{R}^2 \rightarrow \mathbb{C}$. Then,
\begin{equation}
\label{rir0}
\Gamma(A,K)
=2\int_0^1(1-t)\,\big[\Gamma^{A+tK,A+tK,A+tK}(f^{[2]})\big](K,K)\,dt,
\end{equation}
where the latter is an $\Sc^2$-valued Bochner integral. Moreover,
\begin{equation}
\label{rir1}
\tr(\Gamma(A,K)) = 2\int_0^1(1-t)\, \tr \left( \big[\Gamma^{A+tK,A+tK,A+tK}(f^{[2]})\big](K,K) \right)\,dt.
\end{equation}
\end{lemma}

\begin{proof} By Lemma $\ref{existence}$, we let $(A_n)_n$ be a sequence of bounded selfadjoint operators on $\mathcal{H}$ converging resolvent strongly to $A$. Let, for any $n \in \mathbb{N}$, $\varphi_n$ be the function defined by
\begin{equation*}
\varphi_n : t\in \mathbb{R} \mapsto f(A_n+tK)-f(A_n) \in \mathcal{S}^2(\mathcal{H}).
\end{equation*}
The assumptions on $f$ imply, by Theorem $\ref{Formula}$, that $\varphi_n$ is twice differentiable and
\begin{align}
\label{rir2}
\varphi_n^{(2)}(t)=2\big[\Gamma^{A_n+tK,A_n+tK,A_n+tK}(f^{[2]})\big](K,K)
\end{align}
is continuous from $\mathbb{R}$ into $\mathcal{S}^2(\mathcal{H})$. Hence, the integral
\begin{equation*}
2\int_0^1(1-t)\,\big[\Gamma^{A_n+tK,A_n+tK,A_n+tK}(f^{[2]})\big](K,K)\,dt
\end{equation*}
exists in $\mathcal{S}^2(\mathcal{H})$.

Let $\psi$ be a continuous linear functional on $\mathcal{S}^2(\mathcal{H})$. Then,
\begin{align}
\label{rir3}
\frac{d^2}{dt^2}\psi(\varphi_n(t))=\frac{d}{dt}\psi\left(\frac{d}{dt}\varphi_n(t)\right)=\psi(\varphi_n^{(2)}(t)).
\end{align}
Integrating by parts ensures
\begin{align}
\label{rir4}
\psi\left(\varphi_n(1) - \varphi_n(0) -\frac{d}{dt}\varphi_n(t)_{|t=0}\right)
=\int_0^1(1-t)\,\frac{d^2}{dt^2}\psi(\varphi_n(t))\,dt.
\end{align}
Applying the equalities \eqref{rir3} and \eqref{rir2} in \eqref{rir4} gives
\begin{align*}
\psi\left(\varphi_n(1) - \varphi_n(0) - \frac{d}{dt}\varphi_n(t)_{|t=0}\right)
& = 2\int_0^1(1-t)\,\psi\left(\big[\Gamma^{A_n+tK,A_n+tK,A_n+tK}(f^{[2]})\big](K,K)\right)\,dt\\
& = \psi\left(2\int_0^1(1-t)\,\big[\Gamma^{A_n+tK,A_n+tK,A_n+tK}(f^{[2]})\big](K,K)\,dt\right).
\end{align*}
The previous equality holds true for every $\psi$ so we deduce
\begin{align}\label{rir5}
\varphi_n(1) - \varphi_n(0) - \frac{d}{dt}\varphi_n(t)_{|t=0} = 2\int_0^1(1-t)\,\big[\Gamma^{A_n+tK,A_n+tK,A_n+tK}(f^{[2]})\big](K,K)\,dt.
\end{align}

By Corollary $\ref{CoroTaylor}$, we have
\begin{equation*}
\mathcal{S}^2\text{-}\underset{n \rightarrow +\infty}{\lim} \Big(\varphi_n(1) - \varphi_n(0) -\frac{d}{dt}\varphi_n(t)_{|t=0}\Big) = \Gamma(A,K).
\end{equation*}
By Proposition $\ref{Approx1}$, we have
\begin{equation*}
\mathcal{S}^2\text{-}\underset{n \rightarrow +\infty}{\lim} \ \left[\Gamma^{A_n+tK,A_n+tK,A_n+tK}(f^{[2]})\right](K,K) = \left[\Gamma^{A+tK,A+tK,A+tK}(f^{[2]})\right](K,K).
\end{equation*}
Note that for all $n\in \mathbb{N}$ and for all $0 \leq t \leq 1$, $\Gamma^{A_n+tK,A_n+tK,A_n+tK}$ is a contraction by Proposition $\ref{nintegrals}$ so we get the inequality
\begin{equation*}
\|(1-t)\,\big[\Gamma^{A_n+tK,A_n+tK,A_n+tK}(f^{[2]})\big](K,K)\|_2 \leq \|f^{[2]}\|_{\infty} \|K\|_2^2 \leq \dfrac{1}{2} \|f''\|_{\infty} \|K\|_2^2.
\end{equation*}
Hence, applying Lebesgue's dominated convergence theorem in \eqref{rir5} gives
\begin{align}
\Gamma(A,K) = \int_0^1(1-t)\,\big[\Gamma^{A+tK,A+tK,A+tK}(f^{[2]})\big](K,K)\,dt,
\end{align}
which completes the proof of \eqref{rir0}.

We now prove the equality $(\ref{rir1})$. Let
\begin{equation*}
h : t\in [0,1] \mapsto \big[\Gamma^{A+tK,A+tK,A+tK}(f^{[2]})\big](K,K).
\end{equation*}
We have to show that
\begin{equation}\label{trint}
\tr \left( \int_0^1(1-t)h(t)\,dt \right) = \int_0^1(1-t)\, \tr(h(t))\,dt.
\end{equation}
Note that we do not know whether $h$ is measurable as an $\mathcal{S}^1$-valued map, and this is why the proof requires some care.

Let $(e_k)_{k \geq 1}$ be a Hilbertian basis of $\mathcal{H}$. By Bochner integrability of $t \mapsto (1-t)h(t)$ and continuity of the linear functional $T \in \mathcal{S}^2(\mathcal{H}) \mapsto \left\langle Th_1, h_2 \right\rangle$ for any $h_1, h_2 \in \mathcal{H}$,
\begin{align*}
\tr \left( \int_0^1(1-t)h(t)\,dt \right)
& = \sum_{k=1}^{+\infty} \left\langle \left( \int_0^1(1-t)h(t)\,dt \right)e_k, e_k \right\rangle\\
& = \sum_{k=1}^{+\infty} \int_0^1(1-t) \left\langle h(t)e_k, e_k \right\rangle dt.
\end{align*}
For any $n \in \mathbb{N}, n\geq 1$ and $0\leq t \leq 1$, we let
\begin{equation*}
S_n(t) = \sum_{k=1}^n (1-t) \left\langle h(t)e_k, e_k \right\rangle.
\end{equation*}
Then
\begin{equation*}
\tr \left( \int_0^1(1-t)h(t)\,dt \right) = \underset{n \rightarrow +\infty}{\lim} \ \int_0^1 S_n(t)\,dt.
\end{equation*}
Let $P_n$ be the orthogonal projection onto $\text{Span}(e_1, \ldots, e_n)$. Then,
\begin{equation*}
S_n(t) =  (1-t) \tr(P_n h(t))
\end{equation*}
so that, by $(\ref{normineqS1})$ of  Proposition \ref{new}, we have for any $0 \leq t \leq 1$,
\begin{equation*}
|S_n(t)| \leq \|h(t)\|_1 \leq \|a\|_{\infty} \|b\|_{\infty} \|K\|_2^2.
\end{equation*}
Hence, \eqref{rir1} follows from above by applying the dominated convergence theorem.
\end{proof}

For any $\alpha\in\mathbb{C}\setminus\mathbb{R}$ and any integer $m\geq 0$, let $f_{m,\alpha} : \mathbb{R} \rightarrow \mathbb{C}$ be defined by
\begin{equation}\label{fma}
f_{m,\alpha}(x) = \dfrac{1}{(x-\alpha)^m}, \ x\in \mathbb{R}.
\end{equation}
The vector space generated by $f_{m,\alpha}, \alpha\in\mathbb{C}\setminus\mathbb{R}, m\geq 0,$ is the space of rational functions with nonpositive degree and poles off $\mathbb{R}$.

\begin{lemma}\label{rationalexample}\
\begin{enumerate}
\item[$(i)$] Every $f_{m,\alpha}^{[2]}$ defined in \eqref{fma} belongs to the algebraic tensor product $C_b(\mathbb{R}) \otimes C_b(\mathbb{R}) \otimes C_b(\mathbb{R})$.
\item[$(ii)$] Each $\phi \in C_b(\mathbb{R}) \otimes C_b(\mathbb{R}) \otimes C_b(\mathbb{R})$ admits a factorization \eqref{star} for some separable Hilbert space $H$ and some bounded Borel functions $a,b : \mathbb{R}^2 \rightarrow \mathbb{C}$.
\end{enumerate}

\end{lemma}

\begin{proof}
$(i)$ If $m=0$ then $f_{m,\alpha}^{[2]} = 0$. Assume now that $m\geq 1$ and write, for simplicity, $f_m = f_{m,\alpha}$ for a fixed complex number $\alpha\in\mathbb{C}\setminus\mathbb{R}$. We have, for any $x_0, x_1 \in \mathbb{R}$, $x_0 \neq x_1$,
$$
f_m^{[1]}(x_0, x_1) = \dfrac{f_m(x_0)-f_m(x_1)}{x_0-x_1} = \dfrac{1}{(x_0-x_1)(x_0 - \alpha)^m} - \dfrac{1}{(x_0-x_1)(x_1 - \alpha)^m}.
$$
Fixing $x_1 \in \mathbb{R}$ and then taking the partial fraction decomposition of $x_0 \mapsto \dfrac{1}{(x_0-x_1)(x_0 - \alpha)^m}$, we get
$$
\dfrac{1}{(x_0-x_1)(x_0 - \alpha)^m} = \dfrac{1}{(x_0-x_1)(x_1 - \alpha)^m} - \sum_{k=1}^m f_k(x_0)f_{m-k+1}(x_1).
$$
Hence, by continuity,
\begin{equation}\label{examplef1}
f_m^{[1]} = - \sum_{k=1}^m f_k \otimes f_{m-k+1}\in C_b(\mathbb{R}) \otimes C_b(\mathbb{R}).
\end{equation}
Then, an easy computation shows that for any $x_0, x_1, x_2 \in \mathbb{R}$,
\begin{equation*}
f_m^{[2]}(x_0, x_1, x_2) = - \sum_{k=1}^m f_k^{[1]}(x_0,x_1)f_{m-k+1}(x_2).
\end{equation*}
By \eqref{examplef1}, $f_k^{[1]} \in C_b(\mathbb{R}) \otimes C_b(\mathbb{R})$ for any $k\geq 1$ so the latter implies that
$$
f_m^{[2]} \in C_b(\mathbb{R}) \otimes C_b(\mathbb{R}) \otimes C_b(\mathbb{R}).
$$

$(ii)$ Note that the set of functions satisfying \eqref{star} is a vector space. Since an elementary tensor of $C_b(\mathbb{R}) \otimes C_b(\mathbb{R}) \otimes C_b(\mathbb{R})$ satisfies \eqref{star} with $H = \mathbb{C}$, we obtain $(ii)$.
\end{proof}

Below we prove our second main result.

\begin{proof}[Proof of Theorem $\ref{ThmA}$]
Let $f$ satisfy the assumptions of Theorem $\ref{ThmA}$. The first step is to show the existence of a complex measure $\gamma$ defined on the Borel subsets of $\mathbb{R}$, depending only on $A$ and $K$, such that
\begin{align*}
\tr(\Gamma(A,K)) = \int_{\mathbb{R}} f''(v)\, d\gamma(v).
\end{align*}

Let $0 \leq t \leq 1$.
For simplicity, write $A_t = A + tK$. By Proposition $\ref{Simplf}$,
\begin{align*}
\tr\left(\big[\Gamma^{A_t,A_t,A_t}(f^{[2]})\big](K,K)\right)
=\tr\left(\big[\Gamma^{A_t,A_t}(\psi)\big](K)K\right),
\end{align*}
where
\[\psi(s,t)=f^{[2]}(s,t,s).\]
Hence, by $(\ref{rir1})$,
\begin{equation}
\label{lastformula1}
\tr(\Gamma(A,K)) = 2\int_0^1(1-t)\, \tr\left(\big[\Gamma^{A_t,A_t}(\psi)\big](K)K\right)\,dt.
\end{equation}

Let $E^t$ denote the spectral measure on $\R^2$ with values in $\mathcal{B}(\mathcal{S}^2(\mathcal{H}))$ given by
\begin{equation*}
E^t(U) = \Gamma^{A_t,A_t}(\chi_{U}),
\end{equation*}
where $U \subset \mathbb{R}^2$ is a Borel set and $\chi_{U}$ is the characteristic function of $U$. This measure appears in the Birman-Solomyak double operator integral \cite{BS1} and, in particular,
\[E^t(U_1\times U_2)(K)= E^{A_t}(U_1)KE^{A_t}(U_2),\]
for all Borel subsets $U_1,U_2$ of $\R$, every $K\in\Sc^2(\H)$.
By properties of the double operator integral,

\begin{equation*}
\Gamma^{A_t,A_t}(\psi) = \int_{\mathbb{R}^2} \psi(x,y) \, dE^t(x,y).
\end{equation*}
Hence
\begin{equation*}
\tr\left(\big[\Gamma^{A_t,A_t}(\psi)\big](K)K\right) = \int_{\mathbb{R}^2} \psi(x,y) \, d\nu_t(x,y),
\end{equation*}
where $\nu_t$ is the complex measure defined for every Borel subset $U \subset \mathbb{R}^2$ by
\begin{equation*}
\nu_t(U) = \tr \left(\left[\Gamma^{A_t,A_t}(\chi_{U})\right](K)K \right).
\end{equation*}
If $g\in C_0(\mathbb{R}^2)$, then
\begin{align*}
\left\langle \nu_t, g \right\rangle = \int_{\mathbb{R}^2} g(x,y)\, d\nu_t(x,y) = \tr \left(\left[\Gamma^{A_t,A_t}(g)\right](K)K \right)
\end{align*}
is continuous in $t$ by Proposition $\ref{Approx1}$. Hence, the mapping
\begin{equation*}
t \in [0,1] \mapsto \nu_t \in (C_0(\mathbb{R}^2))^*,
\end{equation*}
where $(C_0(\mathbb{R}^2))^*$ is equipped with the $w^*$-topology, is continuous.

Since $(\nu_t)_t$ is bounded in $(C_0(\mathbb{R}^2))^*$, we can define
\begin{equation*}
\nu = 2 \int_0^1 (1-t) \nu_t \, dt \in (C_0(\mathbb{R}^2))^*
\end{equation*}
in the $w^*$-sense, that is, for any $g\in C_0(\mathbb{R}^2)$,
\begin{equation*}
\left\langle \nu, g \right\rangle = 2\int_0^1 (1-t)\left\langle \nu_t, g \right\rangle dt.
\end{equation*}
Moreover, a simple application of the dominated convergence theorem implies that the previous equality holds true whenever $g\in C_b(\mathbb{R}^2)$. By $(\ref{lastformula1})$ we deduce
\begin{align*}
\tr(\Gamma(A,K))
& = 2\int_0^1(1-t)\, \left( \int_{\mathbb{R}^2} \psi(x,y) \ d\nu_t(x,y) \right)\,dt \\
& = \int_{\mathbb{R}^2} \psi(x,y) \, d\nu(x,y).
\end{align*}
By the integral representation \cite[(5.9)]{PSS-SSF},
\begin{align*}
\psi(x,y) = \int_0^1 \lambda f''(\lambda x + (1-\lambda)y) \, d\lambda.
\end{align*}
By a change of variables we get, setting $u=\lambda x+(1-\lambda)y$, $\lambda =\frac{u-x}{x-y}$ and $du=(x-y)d\lambda$,
\begin{equation*}
\psi(x,y) = \dfrac{1}{2} f''(x) \chi_{\left\lbrace x=y \right\rbrace} + \left(\int_{\mathbb{R}} \dfrac{|u-y|}{(x-y)^2} \chi_{[x,y]}(u) f''(u) du \right) \chi_{\left\lbrace x \neq y \right\rbrace},
\end{equation*}
where $\chi_{[x,y]}$ is understood as $\chi_{[y,x]}$ in the case when $y < x$. Let
\begin{equation*}
h(x,y,u)=\dfrac{|u-y|}{(x-y)^2} \chi_{[x,y]}(u) \chi_{\left\lbrace x \neq y \right\rbrace}(x,y).
\end{equation*}
The function $h$ is nonnegative and $\int_{\mathbb{R}} h(x,y,u) \ du = \frac{1}{2}$ for all $(x,y) \in \mathbb{R}^2$. Since $\nu$ is a finite measure, the function $h$ is integrable on $\mathbb{R}^3$ with respect to $d\nu \otimes du$. Hence by Fubini's theorem, the function
\begin{equation*}
\kappa : u \in \mathbb{R} \mapsto \int_{\mathbb{R}^2} h(x,y,u) \, d\nu(x,y)
\end{equation*}
is integrable. Let $\Delta$ denote the diagonal of $\mathbb{R}^2$. Another application of
Fubini's theorem (using the fact that $f''$ is bounded) yields
\begin{align*}
\tr(\Gamma(A,K))
& = \dfrac{1}{2} \int_{\Delta} f''(x) \, d\nu(x,y) + \int_{\mathbb{R}^2} \left( \int_{\mathbb{R}} h(x,y,u)f''(u) du \right) d\nu(x,y) \\
& = \dfrac{1}{2} \int_{\Delta} f''(x) \, d\nu(x,y) + \int_{\mathbb{R}} f''(u) \left( \int_{\mathbb{R}^2} h(x,y,u) \, d\nu(x,y) \right) du \\
& = \dfrac{1}{2} \int_{\Delta} f''(x) \, d\nu(x,y) + \int_{\mathbb{R}} f''(u) \kappa(u)\, du.
\end{align*}
Denote by $\alpha$ the measure defined on Borel subsets $V \subset \mathbb{R}$ by
\begin{equation*}
\alpha(V) = \dfrac{1}{2} \nu( (V \times V) \cap \Delta),
\end{equation*}
by $\beta$ the measure
\begin{equation*}
\beta= \kappa\,du,
\end{equation*}
and let
\[\gamma = \alpha + \beta.\]
The previous computation implies that
\begin{align*}
\tr(\Gamma(A,K))= \int_{\mathbb{R}} f''(v)\, d\gamma(v).
\end{align*}

Let $\eta$ be the Koplienko spectral shift function, so that \eqref{Kop} holds for any rational function with nonpositive degree and poles off $\mathbb{R}$. The last step is to prove that $\gamma$ coincides with $\eta dt$. Let $f_m(x) = \dfrac{1}{(x-i)^m}$ for $m\in\N$. By Lemma \ref{rationalexample}, the functions $f_m$ and their conjugate functions $\overline{f_m}$ satisfy the assumptions of Theorem $\ref{ThmA}$. Hence, for every $m\in\N$,
\begin{equation*}
\int_{\mathbb{R}} f_m''(v)\, d\gamma(t) = \int_{\mathbb{R}} f_m''(t) \eta(t)\, dt
\end{equation*}
and similarly for the conjugate functions $\overline{f_m}''$. Note that the space generated by the $f_m''$ and $\overline{f_m}'', m \in \mathbb{N}, m\geq 1$, is equal to $\text{Span}\left\lbrace f_k, \overline{f_k}, k\geq 3 \right\rbrace$ which is, by Stone-Weierstrass theorem, dense in $C_0(\mathbb{R})$. This implies that
\begin{equation*}
\gamma = \eta\, dt,
\end{equation*}
concluding the proof.
\end{proof}

We now give two examples of classes of functions that satisfy assumptions of Theorem $\ref{ThmA}$.

The first example is the Besov class $B^2_{\infty 1}(\mathbb{R})$. We start by recalling its definition. Let $w_0\in C^{\infty}(\mathbb{R})$ be such that its Fourier transform is supported in $[-2, -1/2] \cup [1/2, 2]$, $\hat{w_0}$ is an even function and $\hat{w_0}(y) + \hat{w_0}(y/2) = 1$ for $1 \leq y \leq 2$. Set $w_n(x) = 2^n w_0(2^nx)$ for $x\in \mathbb{R}, n\in \mathbb{Z}$. Following \cite{Peller2005}, the Besov space is defined as the set
$$
B^2_{\infty 1}(\mathbb{R}) = \Big\lbrace f\in C^2(\mathbb{R}) \ | \ \ \|f''\|_{\infty} + \sum_{n\in \mathbb{Z}} 2^{2n} \|f \ast w_n \|_{\infty} < \infty \Big\rbrace,
$$
equipped with the seminorm
$$
\|f\|_{B^2_{\infty 1}(\mathbb{R})} = \|f''\|_{\infty} + \sum_{n\in \mathbb{Z}} 2^{2n} \|f \ast w_n \|_{\infty}.
$$

We refer, e.g., to \cite{Peller2005} for characterizations of elements of Besov spaces. We prove below that if $f \in B^2_{\infty 1}(\mathbb{R})$, then the divided difference $f^{[2]}$ satisfies the property $(\ref{Hypf})$, and hence Theorem \ref{ThmA}. In this case, we recover \cite[Theorem 4.6]{Peller2005}.

\begin{proposition}\label{Besovfacto}
Let $f\in B^2_{\infty 1}(\mathbb{R})$. Then there exist a separable Hilbert space $H$ and two bounded Borel functions $a : \mathbb{R} \rightarrow H$ and $b : \mathbb{R}^2 \rightarrow H$ such that
\begin{equation*}
f^{[2]}(s,t,u) = \left\langle a(s), b(t,u) \right\rangle,\quad \forall (s,t,u)\in \mathbb{R}^3.
\end{equation*}
In particular, $f^{[2]}$ satisfies $(\ref{Hypf})$.
\end{proposition}

\begin{proof}
If $f\in B^2_{\infty 1}(\mathbb{R})$, then according to \cite[Theorem 5 (ii)]{PS} (or similary, to \cite[Theorem 5.1]{Peller2006}) and its proof, there exist a measure space $(\Omega, \sigma)$ and bounded measurable functions $\alpha, \beta, \gamma$ on $\mathbb{R} \times \Omega$ such that for all $w\in \Omega$,
$$\alpha(\cdot,w), \ \beta(\cdot,w) \ \ \text{and} \ \ \gamma(\cdot,w)$$
are continuous on $\mathbb{R}$ with
\begin{equation}\label{Besovineq}
\int_{\Omega} \| \alpha(\cdot,w)\|_{\infty} \| \beta(\cdot,w)\|_{\infty} \| \gamma(\cdot,w)\|_{\infty} \, d\sigma(w) < \infty,
\end{equation}
and such that for all $(s,t,u) \in \mathbb{R}^3$,
\begin{equation}\label{Besovfactorization}
f^{[2]}(s,t,u) = \int_{\Omega} \alpha(s,w) \beta(t,w) \gamma(u,w) \, d\sigma(w).
\end{equation}
By construction of the measure space $(\Omega, \sigma)$, the space $H = L^2(\Omega, \sigma)$ is separable.

By changing $\Omega$ if necessary, we can assume that $\| \alpha(\cdot,w)\|_{\infty} \|\beta(\cdot,w)\|_{\infty} \|\gamma(\cdot,w)\|_{\infty} >0$ for every $w\in\Omega$.
Let $\delta$ be the measurable function defined on $\Omega$ by
$$\delta(w) = \left(\dfrac{\| \beta(\cdot,w)\|_{\infty} \| \gamma(\cdot,w)\|_{\infty}}{\| \alpha(\cdot,w)\|_{\infty}}\right)^{1/2}, \ w\in \Omega.$$
For $s,t,u \in \mathbb{R}$, let $a(s)$ and $b(t,u)$ be the measurable functions defined on $\Omega$ by
\begin{equation*}
[a(s)](w) = \alpha(s,w) \delta(w), \ [b(t,u)](w) =\overline{\beta(t,w) \gamma(u,w)} \delta(w)^{-1} , \ w\in \Omega.
\end{equation*}

It is straightforward to check that for all $s,t,u \in \mathbb{R}$,
\begin{equation}\label{ineg1Besov1}
\|a(s)\|_H^2 \leq \int_{\Omega} \| \alpha(\cdot,w)\|_{\infty} \| \beta(\cdot,w)\|_{\infty} \| \gamma(\cdot,w)\|_{\infty} \, d\sigma(w)
\end{equation}
and
\begin{equation}\label{ineg1Besov2}
\|b(t,u)\|_H^2 \leq \int_{\Omega} \| \alpha(\cdot,w)\|_{\infty} \| \beta(\cdot,w)\|_{\infty} \| \gamma(\cdot,w)\|_{\infty} \, d\sigma(w).
\end{equation}
Combining the latter with \eqref{Besovineq} confirms that
$$
a\colon\Rdb\to H
\qquad\hbox{and}\qquad
b\colon\Rdb^2\to H
$$
are bounded.

For any $s_0,s \in\Rdb$,
$$
\Vert a(s) -a(s_0)\Vert_{2}^{2} = \int_{\Omega} \delta(w)^2\bigl\vert
\alpha(s,w) -\alpha(s_0,w)\bigr\vert^{2}\, d\sigma(w).
$$
Since $\alpha$ is continuous in the first variable, the uniform estimate \eqref{ineg1Besov1} and Lebesgue's
dominated convergence theorem imply that $a$ is continuous. Similarly, $b$ is continuous. Finally, by \eqref{Besovfactorization},
$$f^{[2]}(s,t,u) = \left\langle a(s), b(t,u) \right\rangle_H,$$
concluding the proof.
\end{proof}

\begin{remark}\label{PellerMS}
The expressions $f(A+K) - f(A)$ and $\dfrac{d}{dt}f(A+tK)_{|t=0}$ may be not well defined bounded operators when $A$ is unbounded and $f$ is not Lipschitz. In this case, as shown by Theorem \ref{Formula} and Corollary \ref{CoroTaylor}, the proper replacement of the undefined difference $f(A+K) - f(A) - \dfrac{d}{dt}f(A+tK)_{|t=0}$ is given by the triple operator integral $\left[\Gamma^{A+K,A,A}(f^{[2]})\right](K,K)$, which is always well defined provided that $f''$ is bounded.

In \cite{Peller2005}, for $f(x) = x^2$, the undefined difference
\begin{equation*}
f(A+K) - f(A) - \dfrac{d}{dt}f(A+tK)_{|t=0}
\end{equation*}
was replaced by $K^2$, which was validated by the fact
\begin{equation}\label{trK}
\tr(K^2) = 2 \int_{\mathbb{R}} \eta(t) \, dt,
\end{equation}
see \cite[Theorem 4.5]{Peller2005}. Since $f^{[2]}(s,t,u) = 1$ for all $(s,t,u) \in \mathbb{R}^3$, the representation \eqref{MOItensor} ensures
$$\left[\Gamma^{A+K,A,A}(f^{[2]})\right](K,K) = K^2.$$
Hence, $\eqref{trK}$ is a consequence of Theorem $\ref{ThmA}$.
\end{remark}

We now consider the class of bounded functions $\phi : \mathbb{R}^3 \rightarrow \mathbb{C}$ such that there exist Hilbert spaces $E$ and $F$ and bounded functions
\begin{equation*}
\alpha : \mathbb{R} \rightarrow E, \theta : \mathbb{R} \rightarrow \mathcal{B}(F,E), \beta : \mathbb{R} \rightarrow F
\end{equation*}
such that
\begin{equation}\label{factophi}
\phi(s,t,u) = \left\langle \alpha(s), \theta(t)\beta(u) \right\rangle,\quad \forall (s,t,u)\in \mathbb{R}^3.
\end{equation}
This class can be identified with the dual space $(\ell^1_{\mathbb{R}} \overset{h}{\otimes} \ell^1_{\mathbb{R}} \overset{h}{\otimes} \ell^1_{\mathbb{R}})^*,$ where $\overset{h}{\otimes}$ is the Haagerup tensor product. See, e.g., \cite{Ruan} for properties of the Haagerup tensor product and \cite[Theorem 9.44]{Ruan} for the description of the dual of the Haagerup tensor product which yields the above identification.
The property $(\ref{factophi})$ naturally arises in the study of completely bounded triple operator integrals. According to \cite[Theorem 4.12]{Coine}, if $A,B,C$ are selfadjoint operators densely defined in a separable Hilbert space $\mathcal{H}$ and $\phi \in (\ell^1_{\mathbb{R}} \overset{h}{\otimes} \ell^1_{\mathbb{R}} \overset{h}{\otimes} \ell^1_{\mathbb{R}})^*$, then the  triple operator integral $\Gamma^{A,B,C}(\phi)$ extends to a completely bounded mapping
\begin{equation*}
\Gamma^{A,B,C}(\phi) : \mathcal{S}^{\infty}(\mathcal{H}) \overset{h}{\otimes} \mathcal{S}^{\infty}(\mathcal{H}) \rightarrow \mathcal{S}^{\infty}(\mathcal{H}),
\end{equation*}
where $\mathcal{S}^{\infty}(\mathcal{H})$ is the space of compact operators on $\mathcal{H}$. We note that an analogous result holds for operator integrals defined on spaces of functions, see \cite[Theorem 3.4]{JTT}.

In the following, we show that a $C^2$-function $f$ with $f^{[2]} \in (\ell^1_{\mathbb{R}} \overset{h}{\otimes} \ell^1_{\mathbb{R}} \overset{h}{\otimes} \ell^1_{\mathbb{R}})^*$ satisfies the property $(\ref{Hypf})$. Note that in the factorization \eqref{factophi}, we do not assume that $E$ and $F$ are separable so the fact that $f^{[2]}$ satisfies the property $(\ref{Hypf})$ requires an explanation.

\begin{proposition}\label{Haagerup}
Let $\phi : \mathbb{R}^3 \rightarrow \mathbb{C}$ be a separately continuous function in $(\ell^1_{\mathbb{R}} \overset{h}{\otimes} \ell^1_{\mathbb{R}} \overset{h}{\otimes} \ell^1_{\mathbb{R}})^*$ such that $\| \phi \|_{(\ell^1_{\mathbb{R}} \overset{h}{\otimes} \ell^1_{\mathbb{R}} \overset{h}{\otimes} \ell^1_{\mathbb{R}})^*} \leq 1$. Then, there exist $(\alpha_m)_{m \geq 1}, (\beta_n)_{n\geq 1}, (\theta_{m,n})_{m,n\geq 1}$ in $C(\mathbb{R})$ such that
$$
\forall s\in \mathbb{R}, \ \sum_m |\alpha_m(s)|^2 \leq 1,
$$
$$
\forall u\in \mathbb{R}, \ \sum_n |\beta_n(u)|^2 \leq 1,
$$
$$
\forall t \in \mathbb{R}, \ \| [\theta_{m,n}(t)]_{m,n \geq 1} \|_{\mathcal{B}(\ell^2)} \leq 1,
$$
and for all $(s,t,u) \in \mathbb{R}^3$,
\begin{equation}\label{factophi2}
\phi(s,t,u) = \left\langle (\alpha_m(s)), [\theta_{m,n}(t)] (\beta_n(u)) \right\rangle_{\ell^2}.
\end{equation}
In particular, if $f \in C^2(\mathbb{R})$ is such that  $f^{[2]} \in (\ell^1_{\mathbb{R}} \overset{h}{\otimes} \ell^1_{\mathbb{R}} \overset{h}{\otimes} \ell^1_{\mathbb{R}})^*$, then $f^{[2]}$ satisfies the property $(\ref{Hypf})$.
\end{proposition}

\begin{proof}
We adapt some ideas from the proof of \cite[Theorem 2.2.4]{APeller}. We can assume in the factorization \eqref{factophi} of $\phi$ that $\alpha$ and $\beta$ have a dense range since we can compose on the left $\alpha$ and $\beta$ by the orthogonal projection onto the closure of their range. Let
\begin{equation*}
K = \overline{\left\lbrace \theta(t)^*e \ | \ t\in \mathbb{R}, e\in E \right\rbrace} \subset F
\end{equation*}
and
\begin{equation*}
H = \overline{\left\lbrace \theta(t)f \ | \ t\in \mathbb{R}, f\in F \right\rbrace} \subset E.
\end{equation*}
Let $P_K : F \rightarrow K$ and $P_H : E \rightarrow H$ be the corresponding orthogonal projections.
Since $\phi$ is continuous with respect to the third variable, we deduce that for all $(s,t) \in \mathbb{R}^2$,
\begin{equation*}
u \mapsto \left\langle \theta(t)^*\alpha(s), \beta(u) \right\rangle
\end{equation*}
is continuous. By boundedness of $\beta$ and by density, this implies that for all $k\in K$,
\begin{equation*}
u \mapsto \left\langle k, \beta(u) \right\rangle
\end{equation*}
is continuous. Let $k\in \left\lbrace P_K\beta(u) \ | \ u \in \mathbb{Q} \right\rbrace^\bot$. Then, for any $s\in \mathbb{Q}$,
\begin{equation*}
\left\langle k,\beta(s) \right\rangle = 0,
\end{equation*}
and by continuity, this equality holds true for any $s\in \mathbb{R}$. Since the range of $\beta$ is dense, this implies that $k=0$. This shows that $K = \overline{\left\lbrace P_K\beta(u) \ | \ u \in \mathbb{Q} \right\rbrace}$, so $K$ is separable. Similarly, we show that for all $h\in H$,
\begin{equation*}
s \mapsto \left\langle \alpha(s), h \right\rangle
\end{equation*}
is continuous and that $H$ is separable.

Define
$$
\tilde{\alpha} = P_H \alpha : \mathbb{R} \rightarrow H, \ \tilde{\beta} = P_K \beta : \mathbb{R} \rightarrow K
$$
and
$$
\tilde{\theta} = P_H \theta_{|K} : \mathbb{R} \rightarrow \mathcal{B}(K,H).
$$
Then, it is easy to check that
$$
\phi(s,t,u) = \left\langle \tilde{\alpha}(s), \tilde{\theta}(t)\tilde{\beta}(u) \right\rangle,\quad \forall (s,t,u)\in \mathbb{R}^3.
$$
Since $\tilde{\alpha}$ and $\tilde{\beta}$ have a dense range, the continuity of $\phi$ in the second variable implies that for any $h \in H$ and any $k\in K$, the function
$$t \mapsto \left\langle h, \tilde{\theta}(t)k \right\rangle$$
is continuous.

Finally, let $(e_m)_{m\geq 1}$ and $(f_n)_{n\geq 1}$ be hilbertian bases of $H$ and $K$, respectively. We set, for any $n,m \geq 1$,
$$
\alpha_m(s) = \left\langle \tilde{\alpha}(s), e_m \right\rangle, \ \beta_n(u) = \left\langle f_n, \tilde{\beta}(u)\right\rangle \ \ \text{and} \ \ \theta_{m,n}(t) = \left\langle e_m, \tilde{\theta}(t) f_n\right\rangle.
$$
The latter implies that $\alpha_m$, $\beta_n$ and $\theta_{m,n}$ are continuous and satisfy \eqref{factophi2}.
\end{proof}

\vskip 1cm
\noindent
{\bf Acknowledgements.}
Research of C.C. and C.L. was supported by the French ``Investissements d'Avenir"
program, project ISITE-BFC (contract ANR-15-IDEX-03). Research of C.C. was also
supported in part by NSF grant DMS-1554456. Research of A.S. was supported
in part by NSF grant DMS-1500704. Research of F.S. was supported by
ARC Discovery grant DP150100920.
\\

\vskip 1cm

\end{document}